\documentclass{amsart}
\usepackage{amsthm,amsfonts,amsmath,amscd,amssymb,latexsym,epsfig}
\usepackage[centredisplay,Postscript=dvitops]{diagrams}
\newcommand{\gl}{{\mathfrak g \mathfrak l}}

\newcommand{\su}{{\mathfrak s  \mathfrak u}}

\newcommand{\g}{{\mathfrak g}}         

\newcommand{\cx}{{\mathbb C}}
\newcommand{\diag}{\operatorname{diag}}

\newcommand{\tr}{\operatorname{tr}}

\newcommand{\im}{\operatorname{Im}}

\newcommand{\Lie}{\operatorname{Lie}}

\newcommand{\ol}{\overline}

\numberwithin{equation}{section}

\newtheorem{theorem}{Theorem}[section]

\newtheorem{lemma}[theorem]{Lemma}

\newtheorem{corollary}[theorem]{Corollary}

\newtheorem{proposition}[theorem]{Proposition}

\theoremstyle{remark}

\newtheorem{remark}[theorem]{Remark}

\newtheorem{example}[theorem]{Example}

\newcommand{\oH}{{\mathbb{H}}}

\newcommand{\oN}{{\mathbb{N}}}

\newcommand{\oP}{{\mathbb{P}}}

\newcommand{\oR}{{\mathbb{R}}}

\newcommand{\sG}{{\mathcal{G}}}   

\newcommand{\sM}{{\mathcal{M}}}   

\newcommand{\sO}{{\mathcal{O}}}

\newcommand{\sS}{{\mathcal{S}}}

\newcommand{\sZ}{{\mathcal{Z}}}

\newcommand{\fL}{{\mathfrak{l}}}

\newcommand{\fS}{{\mathfrak{s}}}

\begin{document}

\title[Slices, Atiyah-Hitchin, and Hilbert schemes of points]{Slices to sums of adjoint orbits, the Atiyah-Hitchin manifold, and Hilbert schemes of points}
\author{Roger Bielawski}
\address{Institut f\"ur Differentialgeometrie\\
Universit\"at Hannover\\ Welfengarten 1\\ D-30167 Hannover}


\begin{abstract} We show that the regular Slodowy slice to the sum of two semisimple adjoint orbits of $GL(n,\cx)$ is isomorphic to the deformation of the $D_2$-singularity if $n=2$, the Dancer deformation of the double cover of the Atiyah-Hitchin manifold if $n=3$, and to the Atiyah-Hitchin manifold itself if $n=4$. For higher $n$, such slices to the sum of two orbits, each having only two distinct eigenvalues, are either empty or biholomorphic to open subsets of the Hilbert scheme of points on of one the above surfaces. In particular, these open subsets of Hilbert schemes of points carry complete hyperk\"ahler metrics. In the case of the double cover of the Atiyah-Hitchin manifold this metric turns out to be the natural $L^2$-metric on a hyperk\"ahler submanifold of the monopole moduli space.  
\end{abstract}

\maketitle

\thispagestyle{empty}

Let $G$ be a compact Lie group with Lie algebra $\g$ and let $(e,h,f)$ be an $\fS\fL(2,\cx)$-triple in $\g^\cx$. The affine subspace $S(f)=f+C(e)$ of $\g^\cx$, where $C(e)$ denotes the centraliser of $e$, is called the Slodowy slice to the nilpotent orbit of $f$ \cite{Slo}. It has the remarkable property of being transverse to any adjoint orbits it meets. Let $M$ be a hyperk\"ahler manifold with a tri-Hamiltonian action of $G$ and let $\mu:M\to \g^\cx$ be the complex moment map for one of the complex structures. It has been shown in \cite{Bie1} (extending results of Kronheimer \cite{Kron} on slices to nilpotent orbits) that, under mild assumptions on $M$, $\mu^{-1}(S(f))$ carries a natural hyperk\"ahler structure. Moreover, as we show in the present paper, the hyperk\"ahler metric on $\mu^{-1}(S(f))$ is complete if the original metric on $M$ is complete.
\par
A remarkable number of hyperk\"ahler manifolds arise as such Slodowy slices to simpler hyperk\"ahler manifolds: 
\begin{itemize} 
\item the moduli space of  $SU(2)$-monopoles of charge $k$ is the regular slice to $T^\ast GL(k,\cx)$ for the $U(k)\times U(k)$-action \cite{Bie1} ("regular slice" means that the nilpotent orbit of $f$ is regular);
\item ALE gravitational instantons are subregular slices to regular semisimple adjoint orbits \cite{Bie1,Bie2} (this has been also rediscovered by Manolescu \cite{Mano} in the case of $A_{2m}$ ALE spaces and by Jackson \cite{Jack} for other ALE spaces of types A and D);
\item ALF gravitational instantons of type $D_k$, $k\geq 3$, are regular slices to regular semisimple adjoint orbits of $SL(k,\cx)$ for the action of $SU(2)\times SU(k-2)\subset SU(k)$ \cite{Cher-Kap};
\item the $D_2$ ALF gravitational instanton is the regular slice to the product of two semisimple adjoint orbits of $SL(2,\cx)$ \cite{Bie2}.  
\end{itemize} 
As seen from these examples, slices to adjoint orbits provide a particularly rich source of hyperk\"ahler manifolds. As far as we know, however, slices to products of two or more orbits have not been investigated, apart from the last example. 
Since the complex moment map of an orbit $\sO$ is the embedding $\sO\hookrightarrow \gl(n,\cx)$, such a slice is the affine variety
$$ \left\{(A_1,\dots,A_k)\in \prod_{i=1}^k \sO_i;\; A_1+\dots +A_k\in S(f)\right\}.
$$
It makes therefore sense to speak about a slice to the {\em sum} of  orbits.
\par
In the present paper we make the first step and consider regular
Slodowy slices to the sum of two adjoint semisimple orbits of $GL(n,\cx)$, each orbit having only two eigenvalues with multiplicities $k_i$ and $l_i$ ($k_i+l_i=n$), $i=1,2$. They are nonempty only if $j=|k_1-l_1|+|k_2-l_2|\leq 2$, and in this case we obtain three series of complete hyperk\"ahler manifolds depending on the value of $j=0,1,2$. These three series turn out to be related to the first three ALF gravitational instantons of type $D_k$, i.e. the Atiyah-Hitchin manifold ($D_0$) if $j=2$, Dancer's deformation of its double cover ($D_1$) if $j=1$, and Hitchin's \cite{Hi} $D_2$-manifold if $j=0$. All three are affine surfaces given by respective equations
$$x^2-zy^2+y=0,$$
$$x^2-zy^2 -1 +\alpha y=0,$$
$$ x^2-zy^2+x+\alpha y+\beta =0.
$$ 
We show that these surfaces and their hyperk\"ahler metrics are realised as slices to the sum of two orbits in $GL(4,\cx)$ for the Atiyah-Hitchin manifold, in $GL(3,\cx)$ for its double cover, and in $GL(2,\cx)$ for the deformation of the $D_2$-singularity (in the last case, the complex structure, but not the metric, has been identified in \cite{Bie2}).
\par
For higher $n$ we identify the nonempty slices as open subsets of Hilbert schemes of points on these surfaces. These open subsets consist of those schemes $Z$ of points for which the projection $\pi$ onto the $z$-coordinate in the above equations induces an isomorphism between $Z$ and its scheme-theoretic image $\pi(Z)$. We call this open subset the {\em Hilbert scheme of points transverse to $\pi$}. The construction of these goes back to Atiyah and Hitchin \cite[Ch. 6]{AH}, who realised spaces of based rational maps on $\oP^1$ as such transverse Hilbert schemes. As observed by Atiyah and Hitchin, if we start with a hyperk\"ahler $4$-manifold, then we can apply the transverse Hilbert scheme construction fibrewise to its twistor space and obtain a new twistor space, which may lead to a new hyperk\"ahler metric. This does not always work: we show (\S\ref{C2pi}) that if we start with flat $\cx^2$ and $\pi(x,y)=xy$, then the new twistor space has no twistor lines. It does however work in case of the $D_0,D_1$ 
and 
$D_2$ surfaces with $\pi(x,y,z)=z$ and we show that the hyperk\"ahler metric on the transverse Hilbert scheme of $m$ points obtained from the slice construction coincides with the one obtained from the fibrewise twistor construction.  In the case of the double cover of the  Atiyah-Hitchin manifold, we can identify this hyperk\"ahler metric
as the natural metric on certain hyperk\"ahler submanifold of the moduli space of $SU(2)$-monopoles of charge $2m$. 
\par
Incidentally, the transverse Hilbert scheme of points on the ALE surfaces of type $A_k$ ($k\geq 1$) and $D_k$ ($k\geq 3$) have been identified as slices to single semisimple adjoint orbits by Seidel and Smith \cite{SS}, Manolescu \cite{Mano} and Jackson \cite{Jack} (the identification is only as complex manifolds but I believe that going through the identification fibrewise on the twistor space will show that the hyperk\"ahler metrics are those arising from the fibrewise Hilbert scheme construction). Hence, there arises the natural question whether the fibrewise transverse Hilbert scheme construction applied to twistor spaces of arbitrary ALF gravitational instantons also produces  complete hyperk\"ahler metrics. We show here that this is the case for type $A_0$, i.e. for the Taub-NUT metric (\S\ref{C2pi}). Since the ALF gravitational instantons of type $A_k$, $k\geq 1$, arise as a Taub-NUT modification of the corresponding ALE spaces, we expect that a combination of the construction for the Taub-NUT metric 
and for the $A_k$ ALE spaces will answer the question positively in this case. For ALF gravitational instantons of type $D_k$ there is a whole unexplored world of possibilities among non-regular slices to sums of two orbits or regular slices to sums of three or more orbits and slices to orbits in other simple Lie groups.
\par
The article is organised as follows. In the next section we recall the hyperk\"ahler slice construction from \cite{Bie1}. In section \ref{reg} we discuss slices to sums of two orbits, each having only two eigenvalues, identify those pairs of orbits for which the regular slice is nonempty and define the three series of hyperk\"ahler manifolds described above (depending on the value of $j\in\{0,1,2\}$). In sections \ref{Ex} and \ref{metric} we identify the slices in the case of orbits in $GL(2,\cx)$, $GL(3,\cx)$, and $GL(4,\cx)$, i.e. the initial members of the three series.        In sections \ref{Hilbpi} and \ref{C2pi} we discuss transverse Hilbert schemes of points and show that there exist complete hyperk\"ahler metrics on such Hilbert schemes of points for the Taub-NUT metric.  In section \ref{HilbDk} we show that nonempty slices to sums of two orbits (each orbit having only two eigenvalues) in $GL(n,\cx)$ with $n>4$ are transverse Hilbert schemes of points on the $D_0$, $D_1$ or $D_2$ surface 
(depending on whether the value of $j$ defined above is $2$, $1$ or $0$). We also show there that the complete hyperk\"ahler metrics on these, which arise from the hyperk\"ahler slice construction, have twistor spaces obtained via the fibrewise transverse Hilbert scheme construction. Finally, in the appendix, we show that a hyperk\"ahler slice to a complete hyperk\"ahler manifold is again complete.

\medskip

{\bf Acknowledgement.} This research has been motivated by Tam\'as Hausel pointing out to the author that the regular slice to the sum of two minimal adjoint orbits in $SL(3,\cx)$ is a complex surface and conjecturing that it should be the $D_1$-manifold. In fact, the hyperk\"ahler spaces studied in this paper are also examples of hyperk\"ahler metrics recently constructed  by Tam\'as Hausel, Michael Wong, and Dimitri Wyss  \cite{HWW} on open de Rham spaces of irregular connections on trivial bundles on the projective line. I warmly thank Tam\'as Hausel for the discussions we had during the meeting {\em Metric and Analytic Aspects of Moduli Spaces} at the Isaac Newton Institute in July/August 2015. I also thank the organisers of that meeting for the invitation, a superb scientific programme, and a stimulating research atmosphere, and the Isaac Newton Institute itself for great working conditions.

\section{Slodowy slices and hyperk\"ahler metrics\label{slices}}

Let $\g$ be a compact Lie algebra and $\rho:\su(2)\to\g $ a homomorphism of Lie algebras. We extend $\rho$ to a homomorphism of complexified Lie algebras and denote by $(e,h,f)$ the resulting $\fS\fL(2,\cx)$-triple. The {\em Slodowy slice} \cite{Slo} corresponding to $\rho$ is the affine subspace $S(\rho)=f+C(e)$ of $\g^\cx$, where $C(e)$ denotes the centraliser of $e$. It has the remarkable property of being transverse to any adjoint orbits it meets.
\par
As shown in \cite{Bie1}, the manifold $S(\rho)\times G^\cx$, where $G$ is a compact Lie group with Lie algebra $\g$, carries a natural  hyperk\"ahler metric. It is described (see Appendix) as the natural $L^2$-metric on the moduli space of $\g$-valued solutions to Nahm's equations on the interval $(0,1]$ (rescaling the interval corresponds to rescaling the metric) with simple poles having residues determined by $\rho$ at $t=0$. Two solutions are identified if they differ by a gauge transformation which is identity at both $t=0$ and $t=1$. This hyperk\"ahler metric on $S(\rho)\times G^\cx$ admits a free tri-hamiltonian (in particular isometric) action of $G$ (given by gauge transformations with arbitrary values at $t=1$) and an isometric $SO(3)$-action rotating the complex structures. In particular, each complex structure of this hyperk\"ahler manifold is that of $S(\rho)\times G^\cx$. The completeness of the hyperk\"ahler metric on $S(\rho)\times G^\cx$ was not shown in \cite{Bie1} - an omission which we 
rectify in the appendix to this paper:
\begin{theorem} The natural hyperk\"ahler metric on $S(\rho)\times G^\cx$ is complete.\label{complete}
\end{theorem}

Let now $M$ be a hyperk\"ahler manifold with a tri-Hamiltonian action of $G$. We can then form the hyperk\"ahler quotient of $M\times S(\rho)\times G^\cx$ by $G$ which is always smooth and, in view of the above theorem, complete, if the metric on $M$ is complete. Moreover, as shown in \cite{Bie1}, the complex structure of this hyperk\"ahler quotient is, under mild assumptions, easily identified:
\begin{theorem}\cite{Bie1} Let $(M,g,I_1,I_2,I_3)$ be a hyperk\"ahler manifold with a tri-hamiltonian action of a compact Lie group $G$. Let $I$ be one of the complex structures and $\mu:M\to \g^\cx$ the corresponding $I$-holomorphic moment map (i.e. $\mu=\mu_2+i\mu_3$ if $I=I_1$). Assume that the action of $G$ extends to an $I$-holomorphic action of $G^\cx$ and that $M$ admits a global K\"ahler potential for the K\"ahler form $g(I\cdot,\cdot)$ which is bounded below on each $G^\cx$-orbit. Then the hyperk\"ahler quotient of $M\times S(\rho)\times G^\cx$ by $G$ is biholomorphic, with respect to the complex structure $I$, to $\mu^{-1}\bigl(S(\rho)\bigr)$.\hfill $\Box$ \label{theorem}
\end{theorem}

We shall refer to the hyperk\"ahler quotient of $M\times S(\rho)\times G^\cx$ as a {\em hyperk\"ahler slice to $M$}. In the current paper we shall be interested in the case $G=U(n)$, $(e,h,f)$ a regular $\fS\fL(2,\cx)$-triple and $M$ a product of flat space and adjoint $GL(n,\cx)$-orbits with their Kronheimer-Biquard-Kovalev (KBK) hyperk\"ahler structure \cite{Kr,Kron, Biq,Ko}. 
In the case of slices to adjoint orbits, we need to be careful the identification of the complex-symplectic and hyperk\"ahler quotients described in the above theorem, since the K\"ahler form of the KBK metric does not need to admit a global K\"ahler potential\footnote{KBK-metrics on a given complex adjoint orbit are parameterised by an element $\tau_1$ of certain abelian subalgebra of $\g$. Only the KBK-metric with $\tau_1=0$ admits a global K\"ahler potential.} (it may define a nonzero $H^{1,1}$-cohomology class).
The following proposition allows us to identify the complex structure of slices to arbitrary semisimple adjoint orbits.
\begin{proposition} Let $M$ be a hyperk\"ahler quotient a finite-dimensional quaternionic vector space $V$ by a closed subgroup $H$ of $Sp(V)$, and let $G$ be a closed subgroup of $Sp(V)$ commuting with $H$. Let $I$ be one of the complex structures and $\nu:V \to \frak{h}^\cx$ the $I$-holomorphic moment map. Assume that $H$ acts freely on $\nu^{-1}(\tau)$ and
the $H^\cx$-orbits in $\nu^{-1}(\tau)$ are closed, where $\tau$ is the value used to obtain $M$. Then the hyperk\"ahler quotient of $M\times S(\rho)\times G^\cx$ by $G$ is biholomorphic, with respect to the complex structure $I$, to $\mu^{-1}\bigl(S(\rho)\bigr)$, where $\mu:M\to \g^\cx$ is the corresponding $I$-holomorphic moment map.
\end{proposition}
\begin{proof} The hyperk\"ahler quotient of $M\times S(\rho)\times G^\cx$ by $G$ is isomorphic to the hyperk\"ahler quotient of $V\times S(\rho)\times G^\cx$ by $G\times H$. The  statement is equivalent to saying that this second hyperk\"ahler quotient is biholomorphic, with respect to the complex structure $I$, to the complex-symplectic quotient of $V\times S(\rho)\times G^\cx$ by $G^{\cx}\times H^{\cx}$. The assumptions imply that all $G^{\cx}\times H^{\cx}$ orbits in the level set $Z$ of the complex moment map are closed and that $G\times H$ acts freely there. Moreover $V$ has a proper K\"ahler potential for the complex structure $I$. It follows, as in \cite[\S4]{Bie1}, that $Z$ has a global K\"ahler potential which is proper on each $G^\cx\times H^\cx$-orbit. The arguments in \cite[\S6.5.2]{DoKr} imply now that every  $G^\cx\times H^\cx$-orbit in $Z$ is stable and hence the hyperk\"ahler and complex-symplectic quotients coincide.
\end{proof}

Hyperk\"ahler metrics on  adjoint orbits of $GL(n,\cx)$ can be obtained, as shown by Nakajima \cite{Nak2} (see also \cite{KS}),  as hyperk\"ahler quotients of a quaternionic vector space by a product of unitary groups, and the assumption about the closedness of $H^\cx$-orbits is fulfilled in the case of semisimple orbits.
Thus we conclude the above result holds for any product $\sO_1\times\dots \times \sO_k$ of semisimple adjoint orbits of $GL(n,\cx)$ (it also holds for nilpotent orbits, thanks to Theorem \ref{theorem}). This means that the hyperk\"ahler slice to this product is biholomorphic to the affine variety
$$\left\{(A_1,\dots,A_k)\in \prod_{i=1}^k \sO_i;\; A_1+\dots +A_k\in S(\rho)\right\}.
$$

As mentioned above, we shall only consider the regular slice, i.e. the case where $\rho$ defines an irreducible representation of $\su(2)$. In this case, the Slodowy slice is isomorphic to the set $\sS_n$ of companion matrices
\begin{equation}\begin{pmatrix} 0 & \dots & \dots & 0  & s_n\\ 1 & \ddots & & 0 & s_{n-1}\\ \vdots & \ddots &\ddots & \vdots &\vdots\\ 0 &  \dots &\ddots & 0&  s_2\\  0 & \dots &\dots & 1 & s_1\end{pmatrix}.\label{S}\end{equation}
We shall denote the manifold $\sS_n\times GL(n,\cx)$ with its natural hyperk\"ahler metric by $N_n$. We shall need later a description of the twistor space of
$N_n$, which has already been found in \cite{BieGM} (see the proof of Proposition 6.1 there). It is given by gluing two copies of $\cx\times \sS_n\times GL(n,\cx)$ with coordinates $(\zeta,S,g)$ and $(\tilde{\zeta},\tilde S,\tilde{g})$ via
\begin{equation}
\tilde{\zeta}=\zeta^{-1},\quad\tilde S= D(\zeta) S D(\zeta)^{-1}/\zeta^{2},\quad \tilde g=g \exp(-S/\zeta)  D(\zeta)^{-1},
\label{twistor}
\end{equation}
where $D(\zeta)=\diag\bigl(\zeta^{-n+1},\zeta^{-n+3},\dots,\zeta^{n-1}\bigr)$.

\section{Regular slices to sums of two orbits\label{reg}}

We consider two semisimple adjoint orbits $\sO_1,\sO_2$ of $GL(n,\cx)$ both having only two different eigenvalues (i.e. orbits which are also complex symmetric spaces). After translating by
multiples of the identity matrix we can assume that $\sO_1$ and $\sO_2$ have eigenvalues $\mu_1,-\mu_1$ with multiplicities $k_1,l_1$ and $\mu_2,-\mu_2$ with multiplicities $k_2,l_2$,
respectively. Thus $A^2=\mu_i^2$ for $A\in \sO_i$ and matrices $A$ satisfying this equation fall into different orbits according to the value of $\tr A$.
Let us write $d_1=k_1-l_1$ and $d_2=k_2-l_2$, so that $\tr A=d_1 \mu_1$, $\tr B=d _2 \mu_2$. The numbers $d_1$ and $d_2$ have the same parity as $n$.
\par
The regular Slodowy slice to the sum of these two orbits is the variety 
\begin{equation} \bigl(\sO_1+\sO_2)\cap \sS_n=\left\{(A,B)\in \sO_1\times \sO_2;\; A+B\in \sS_n\right\},\label{var}\end{equation}
where $\sS_n$ is the set of companion matrices \eqref{S}. As described in the previous section, this variety has a natural hyperk\"ahler structure and it arises a complex-symplectic quotient by a free action of a Lie group and is therefore smooth. Moreover, were it nonempty, it would have dimension 
\begin{equation}(n^2-k_1^2-l_1^2)+(n^2-k_2^2-l_2^2)+n-n^2=n-\frac{(d_1)^2}{2}-\frac{(d_2)^2}{2}.\label{dim}\end{equation}
However:
\begin{proposition} Suppose that $|d_1|+|d_2|> 2$. Then the variety \eqref{var} is empty. 
\label{empty}\end{proposition}
\begin{proof} We can replace $\sS_n$ by $f+C(e)$, i.e. by matrices with $s_{ij}=0$  if $i>j+1$, $s_{i+1,i}=1$, and $s_{i,i+r}$ depending only on $r$. In particular $A+B\in f+C(e)$ if an only if $(A+aI)+(B+bI)\in f+C(e)$ for any $a,b\in \cx$. Without loss of generality, we can assume that $k_i\geq l_i$, $i=1,2$. The matrices $A-\mu_1 I$ and $B-\mu_2 I$ have ranks $\leq l_1$, $\leq l_2$, respectively. Since their sum has rank $\geq n-1$, the conclusion follows. 
\end{proof}
\begin{remark} From the point of view of hyperk\"ahler geometry, the fact that \eqref{var} is empty in the case $\frac{(d_1)^2}{2}+\frac{(d_2)^2}{2}\leq n$ and $|d_1|+|d_2|> 2$ is interesting. It means that $N_n\times \sO_1\times\sO_2$ is an example of a complete hyperk\"ahler manifold with a tri-Hamiltonian action of $SU(n)$, such 
the image $W$ of the hyperk\"ahler moment map  is open, but $0\not \in W$.  I am not aware of any other examples of this kind.
\end{remark}

The variety \eqref{var} can therefore be nonempty only if $|d_1|+|d_2|\leq 2$. Since changing $\mu_i$ to $-\mu_i$ is equivalent to replacing $d_i$ with $-d_i$, we shall not lose any generality in assuming that $d_i\geq 0$, $i=1,2$. Similarly, we shall not lose any any generality in assuming that $d_1\geq d_2$. Furthermore, the $d_i$ have the same parity as $n$, and hence there are only three possibilities:
\begin{itemize}
\item $n=2m$ and $d_1=d_2=0$. We shall denote the corresponding variety by $D_{2,m}(\mu_1,\mu_2)$.
\item $n=2m+1$ and $d_1=d_2=1$. The corresponding variety will be denoted by $D_{1,m}(\mu_1,\mu_2)$.
\item $n=2m+2$ and $d_1=2$, $d_2=0$. The corresponding variety will be denoted by $D_{0,m}(\mu_1,\mu_2)$.
\end{itemize}
The reasons for this notation will become apparent. We shall also see that these varieties are nonempty for each $m\geq 1$, and therefore the dimension calculation \eqref{dim} gives $2m$ as the dimension of each of them. Since the orbits $\sO_1$ and $\sO_2$ involved in the definition of $D_{2,m}(\mu_1,\mu_2)$
do not change if we change the sign of $\mu_1$ or of $\mu_2$, we conclude that $D_{2,m}(\mu_1,\mu_2)$ depends only on $\mu_1^2$ and $\mu_2^2$. Similarly $D_{0,m}(\mu_1,\mu_2)=D_{0,m}(\mu_1,-\mu_2)$. We shall see later that $D_{1,m}(\mu_1,\mu_2)$ depends only on $\mu_1-\mu_2$ and $D_{0,m}(\mu_1,\mu_2)$ is independent (up to an affine isomorphism)  of $\mu_1$ and $\mu_2$.

\subsection{Characteristic polynomial of $A+B$}

If $(A,B)$ belong to the variety \eqref{var}, then $S=A+B$ is a companion matrix \eqref{S}. Recall that
the characteristic polynomial $P(z)=\det(z-S)$ of a companion matrix is equal to $z^n-\sum_{i=1}^n s_iz^{n-i}$ and $S$ is the matrix of the multiplication by $z$ on $\cx[z]/(P(z))$ in the basis $1,z,\dots,z^{n-1}$.
\par
We shall need the following algebraic result.
\begin{proposition} Let $A$ and $B$ be two $n\times n$ matrices satisfying $A^2=\mu_1^2$, $B^2=\mu_2^2$, $\tr A=d_1\mu_1$, $\tr B=d_2\mu_2$, $\mu_1,\mu_2\neq 0$, $d_1\geq d_2\geq 0$. If $S=A+B$ is a regular matrix, then the characteristic polynomial of $S$ has the form:
\begin{equation} \left(z-(\mu_1+\mu_2)\right)^{p} \left(z-(\mu_1-\mu_2)\right)^{q}\prod_{i=1}^r(z^2-x_i),
\label{char}\end{equation}
where $p=\frac{d_1+d_2}{2}$, $q=\frac{d_1-d_2}{2}$, and $x_i\in \cx$. 
\label{alg}
\end{proposition}
\begin{proof}
We observe that
$$ \mu_1^2=A^2=(S-B)^2=S^2-SB-BS+\mu_2^2,$$
which we can write as
$$ S(\frac{1}{2}S-B)+(\frac{1}{2}S-B)S=\mu_1^2-\mu_2^2.$$
We abbreviate $Y=\frac{1}{2}S-B$ and $\tau=\mu_1^2-\mu_2^2$, so that the equation we are concerned with is 
\begin{equation} SY+YS=\tau.\label{anti}\end{equation} 
\begin{lemma} If $\lambda\neq 0$ is an eigenvalue of $S$ and the corresponding eigenvector $u$ satisfies $Yu\neq \frac{\tau}{2\lambda}u$, then $-\lambda$ is also an eigenvalue of $S$ with the same algebraic multiplicity as $\lambda$.
\end{lemma}
\begin{proof} 
Suppose that $Su=\lambda u$. Then $ SYu=-\lambda Yu+\tau u$ and it follows that if $\lambda\neq 0$, then 
$$S\left(Yu-\frac{\tau}{2\lambda} u\right)=-\lambda\left(Yu-\frac{\tau}{2\lambda} u\right).$$
By assumption $-\lambda$ is also an eigenvalue of $S$ with eigenvector $Yu-\frac{\tau}{2\lambda}u$.
\par
Since $S$ is regular, the geometric multiplicities of $\lambda$ and $-\lambda$ are both equal to $1$.
Let $m$ be the algebraic multiplicity of $\lambda$,  i.e. there is vector $v$ such that $(S-\lambda)^{m-1}v\neq 0$ and $(S-\lambda)^{m}v=0$. Denote by $E_\lambda$ the kernel of $(S-\lambda)^m$, i.e. the subspace spanned by $v,(S-\lambda)v,\dots, (S-\lambda)^{m-1}v$. The $\lambda$-eigenspace is generated by $u=(S-\lambda)^{m-1}v$ and the $-\lambda$-eigenspace by $Yu-\frac{\tau}{2\lambda} u$. Consider first the case $\tau=0$. Then $Y(S-\lambda)^{m-1}v$ is an $-\lambda$-eigenvector and, since $Y(S-\lambda)^{m-1}v=(-S-\lambda)^{m-1}Yv$, this eigenvector belongs to $\im (S+\lambda)^{m-1}$. It follows that the algebraic multiplicity of $-\lambda$ is at least as large as that of $\lambda$. By symmetry the two multiplicities are equal.
\par
If $\tau\neq 0$, then using the fact that $S^2$ commutes with $Y$, we have for any $k\leq m$
\begin{multline*}(S+\lambda)^k Yv=\sum \begin{pmatrix} k\\ i\end{pmatrix} S^i\lambda^{k-i}Yv= \sum \begin{pmatrix} k\\ i\end{pmatrix} Y(-S)^i\lambda^{k-i}v +\\+\tau \sum_{i\equiv 1\mod 2} \begin{pmatrix} k\\ i\end{pmatrix}\lambda^{k-i}v= Y(-S+\lambda)^kv+\tau \frac{(1+\lambda)^k-(1-\lambda)^k}{2}v
\end{multline*}
Setting $k=m$, we obtain, in particular,  $v\in \im (S+\lambda)^m$. The $-\lambda$-eigenvector $Y(S-\lambda)^{m-1}v-\frac{\tau}{2\lambda}(S-\lambda)^{m-1}v$ can be then written, thanks to the above formula, as
$$(-S-\lambda)^{m-1}Yv +(-1)^m\tau \frac{(1+\lambda)^{m-1}-(1-\lambda)^{m-1}}{2} v -\frac{\tau}{2\lambda} (S-\lambda)^{m-1}v,$$
and so it belongs $\im (S+\lambda)^{m-1}$. Thus the algebraic multiplicities are equal also in this case.
\end{proof}

We now consider the remaining case $Yu=\frac{\tau}{2\lambda}u$.
\begin{lemma} If $\lambda\neq 0$ and $u$ is a $\lambda$-eigenvector of $S$ and $Yu=\frac{\tau}{2\lambda}u$, then $u$ is a common eigenvector of $A$ and $B$, and consequently $\lambda=\pm\mu_1\pm\mu_2$.\label{common}
\end{lemma}
\begin{proof} From assumption, $u$ is a common eigenvector of $S=A+B$ and $Y=(A-B)/2$, i.e. $u$ is a common eigenvector of $A$ and $B$, with eigenvalues $\frac{1}{2}\left(\lambda+\frac{\tau}{\lambda}\right)$ and $\frac{1}{2}\left(\lambda-\frac{\tau}{\lambda}\right)$, respectively. These must be $\pm \mu_1$ and $\pm\mu_2$, so that such a $\lambda$ must be equal to one of the four values $\pm \mu_1\pm\mu_2$. \end{proof}
Finally, for $\lambda=0$ we have
\begin{lemma} Suppose that $0$ is an eigenvalue of $S$. If $\tau\neq 0$ then its algebraic multiplicity is even. If $\tau=0$, then $A$ and $B$ have a common eigenvector (with eigenvalues summing up to zero). \label{zero}
\end{lemma}
\begin{proof}
Suppose that $\tau\neq 0$ and a nonzero $v$ satisfies $S^mv=0$ for an odd $m$. Then $S^m Yv=-YS^m v + \tau v=\tau v$, and so $Yv\neq 0$ and $S^{m+1} Yv=0$. Thus the algebraic multiplicity of $\lambda=0$ is even. If $\tau=0$ and $Su=0$, then $SYu=0$, and so $Yu=\rho u$ for some $\rho$, since the geometric multiplicity of each eigenvalue is $1$.
\end{proof}

Taken together, these three lemmata imply that the characteristic polynomial of $S$ has the form
\begin{equation*}  \left(z-\delta_1(\mu_1+\mu_2)\right)^{m_1} \left(z-\delta_2(\mu_1-\mu_2)\right)^{m_2}\prod_{i=1}^r(z^2-x_i),
\end{equation*}
where $\delta_j=\pm 1$. The trace of $S$ is then 
$\delta_1m_1(\mu_1+\mu_2)+ \delta_2m_2(\mu_1-\mu_2)$. On the other hand 
$$\tr S=d_1\mu_1+d_2\mu_2=\frac{d_1+d_2}{2}(\mu_1+\mu_2) + \frac{d_1-d_2}{2}(\mu_1-\mu_2).$$
According to Proposition \ref{empty}, we have only three possibilities (under the assumption $d_1\geq d_2$): $d_1=d_2=0$, $d_1=d_2=1$, or $d_1=2,d_2=0$. 
Comparing the two formulae for the trace proves  the proposition.
\end{proof}

\subsection{The varieties $D_{k,m}(\mu_1,\mu_2)$\label{nonvan}}

We return to the case when \eqref{var} is nonempty, i.e. to the varieties $D_{k,m}(\mu_1,\mu_2)$, $k=0,1,2$. Their elements can be alternatively described as pairs of matrices $(S,Y)$, where $S$ is of the form \eqref{S}, satisfying the equations 
\begin{equation} SY+YS=\mu_1^2-\mu_2^2,\quad (Y-S/2)^2=\mu_2^2,\label{SY}\end{equation}
together with $\tr S=d_1\mu_1+d_2\mu_2$, $\tr Y=\frac{1}{2}(d_1\mu_1-d_2\mu_2)$, where $(d_1,d_2)=(0,0)$ if $k=2$, $(d_1,d_2)=(1,1)$ if $k=1$, and $(d_1,d_2)=(2,0)$ if $k=0$.
\par
 Proposition \ref{alg} implies that the characteristic polynomial $P(z)$ of $S$ is of the form
\begin{equation*} P(z)=(z-\mu_1-\mu_2)^{\epsilon_1}(z-\mu_1+\mu_2)^{\epsilon_2}Q(z),\label{QQ}\end{equation*}
where $(\epsilon_1,\epsilon_2)$ is equal to $(0,0)$ if $k=2$, to $(1,0)$ if $k=1$, and to $(1,1)$ if $k=0$. The polynomial $Q(z)$ has degree $2m$ with coefficients of all odd powers of $z$ equal to zero. 
\par
If $(\epsilon_1,\epsilon_2)\neq (0,0)$, we can simplify the form of $S$ and $Y$:
\begin{proposition} (i) The affine variety $D_{1,m}(\mu_1,\mu_2)$ is canonically isomorphic to the variety of pairs of matrices $(S,Y)$ satisfying equations \eqref{SY} and having the following form:
\begin{equation}S=\begin{pmatrix} S_0 & 0\\ e & \mu_1+\mu_2\end{pmatrix},\quad Y=\begin{pmatrix} Y_0 & 0 \\ v & (\mu_1-\mu_2)/2\end{pmatrix},
\label{D1m}\end{equation}
where $S_0$ is the companion matrix \eqref{S} to $Q(z)$ (in particular all $s_{2i+1}$ are equal to zero) and $e=(0,\dots,0,1)$.\\
(ii) The affine variety $D_{0,m}(\mu_1,\mu_2)$ is canonically isomorphic to the variety of pairs of matrices $(S,Y)$ satisfying equations \eqref{SY} and having the following form:
\begin{equation}S=\begin{pmatrix} S_0 & 0 & 0\\ e & \mu_1+\mu_2 & 0\\ e & 0 & \mu_1-\mu_2\end{pmatrix},\quad Y=\begin{pmatrix} Y_0 & 0 & 0\\ v_1 & (\mu_1-\mu_2)/2 & 0\\ v_2 & 0 & (\mu_1+\mu_2)/2 \end{pmatrix},
\label{D0m}\end{equation}
where $S_0$ and $e$ are as in case (i).
\label{can}\end{proposition}
\begin{proof} The affine isomorphism is given by the change of basis of $\cx[z]/(P(z))$ from $1,z,\dots, z^{2m}$ to
$ 1,\dots,z^{2m-1}, Q(z)$  in case (i), and to $1,\dots, z^{2m-1}, (2\mu_2)^{-1}(z-\mu_1+\mu_2)Q(z), -(2\mu_2)^{-1}(z-\mu_1-\mu_2)Q(z)$ in case (ii).  Since $S$ corresponds to multiplication by $z$ on $\cx[z]/(P(z))$, it has in both cases the stated form in the new basis. Furthermore, Lemmata \ref{common}, \ref{zero}, and equations \eqref{SY} imply that $Y$ also has the stated form. 
\end{proof}

Observe now that the blocks $S_0$ and $Y_0$ in the above proposition still satisfy equations \eqref{SY} and $\tr S_0=\tr Y_0=0$.  These blocks are therefore elements of the variety $D_{2,m}(\mu_1,\mu_2)$. Thus $(S,Y)\mapsto (S_0, Y_0)$ defines  canonical holomorphic maps from $D_{1,m}(\mu_1,\mu_2)$  or $D_{0,m}(\mu_1,\mu_2)$ to $D_{2,m}(\mu_1,\mu_2)$. Let us write $\phi$ for this holomorphic map from $D_{1,m}(\mu_1,\mu_2)$ to $D_{0,m}(\mu_1,\mu_2)$. In the case of $D_{0,m}$ we can associate to $(S,Y)$ the minor matrices
\begin{equation} S_1=\begin{pmatrix} S_0 & 0\\ e & \mu_1+\mu_2\end{pmatrix},\quad Y_1=\begin{pmatrix} Y_0 & 0\\ v_1 & (\mu_1-\mu_2)/2 \end{pmatrix},
\label{D01}\end{equation}
or 
\begin{equation} S_2=\begin{pmatrix} S_0 & 0\\ e & \mu_1-\mu_2\end{pmatrix},\quad Y_2=\begin{pmatrix} Y_0 & 0\\ v_1 & (\mu_1+\mu_2)/2 \end{pmatrix},
\label{D02}\end{equation}
and so obtain canonical holomorphic maps
\begin{equation} \phi_1:D_{0,m}(\mu_1,\mu_2)\to D_{1,m}(\mu_1,\mu_2), \quad \phi_2:D_{0,m}(\mu_1,\mu_2)\to D_{1,m}(\mu_1,-\mu_2).
\label{phi12}\end{equation}
We have a commutative diagram:

\medskip 

\begin{diagram}& & D_{2,m}(\mu_1,\mu_2)=D_{2,m}(\mu_1,-\mu_2) & &\\&
\ruTo(2,2)^\phi &  &\luTo(2,2)^\phi &\\
D_{1,m}(\mu_1,\mu_2) & & & & D_{1,m}(\mu_1,-\mu_2)\\
& \luTo(2,2)^{\phi_1} & &\ruTo(2,2)^{\phi_2}& \\ &&
D_{0,m}(\mu_1,\mu_2)=D_{0,m}(\mu_1,-\mu_2) & &
\end{diagram}

\medskip

\begin{remark} Observe that a pair $(S,Y)$ of the form \eqref{D0m} satisfies equations \eqref{SY} if and only if both pairs $(S_1,Y_1)$ and $(S_2,Y_2)$ satisfy these equations. Thus $D_{0,m}(\mu_1,\mu_2)$ is the fibred product $D_{1,m}(\mu_1,\mu_2)\times_{D_{2,m}(\mu_1,\mu_2)}D_{1,m}(\mu_1,-\mu_2)$.
\label{pullback}\end{remark}

\section{The case $m=1$ - complex structures\label{Ex}}

We are going to identify the surfaces $D_{2,1}$, $D_{1,1}$ and $D_{0,1}$, i.e. we consider the variety \eqref{var} for $n=2,3,4$. 

\medskip

{\bf n=2.} (cf. \cite{Bie2}) In this case $\tr A=0$ and $\tr B=0$, so that both $S$ and $Y$ are also traceless.
Write 
$$ S=\begin{pmatrix} 0 & x \\ 1 & 0\end{pmatrix}, \quad Y=\begin{pmatrix} a & b \\ c & -a\end{pmatrix}.$$
The equation $SY+YS=\tau$ reduces to $b+cx=\tau$, and the equation $\mu_2^2=B^2=(Y-S/2)^2$ becomes then
$$ a^2-\bigl(x\bigl(c+\frac{1}{2}\bigr)-\tau\bigr)\bigl(c-\frac{1}{2})=\mu_2^2,$$
or
\begin{equation} a^2-xc^2+\frac{1}{4}x+(\mu_1^2-\mu_2^2)c-\frac{1}{2}(\mu_1^2+\mu_2^2)=0,\label{D2}\end{equation}
which is a deformation of the $D_2$-singularity.

\medskip

{\bf n=3.} We can assume that $A\sim (\mu_1,\mu_1,-\mu_1)$ and $B\sim (\mu_2,\mu_2,-\mu_2)$. Proposition \ref{can} implies that our variety is isomorphic to the variety of pairs $(S,Y)$ satisfying \eqref{SY} and of the form  
\begin{equation} S=\begin{pmatrix} 0 & x & 0\\ 1 & 0 & 0\\ 0 & 1 & \mu_1+\mu_2\end{pmatrix},\quad Y=\begin{pmatrix} a & b & 0\\ c & -a & 0\\ y & z & (\mu_1-\mu_2)/2\end{pmatrix}.\label{SY3}\end{equation}
The equation $SY+YS=\tau$ gives again $b+cx=\tau$ and two more equations:
\begin{equation} c+(\mu_1+\mu_2)y +z=0,\label{c}\end{equation}
\begin{equation}-a+ (\mu_1+\mu_2)z+yx+ (\mu_1-\mu_2)/2=0.\label{a}\end{equation}
These allow us to express $a,b$ and $c$ as functions of $y$ and $z$ and substituting into $\mu_2^2=B^2=(Y-S/2)^2$ gives the coordinate ring of the variety \eqref{var},  which is easily seen to be determined by a single equation
\begin{equation} y^2x-z^2+\frac{1}{4} +(\mu_1-\mu_2)y=0.\label{D1}\end{equation}
This is the Dancer deformation of the $D_1$-manifold \cite{Dan}.

\medskip

{\bf n=4.} This time  $A\sim (\mu_1,\mu_1,\mu_1,-\mu_1)$ and $B\sim (\mu_2,\mu_2,-\mu_2,-\mu_2)$. Proposition \ref{can} implies that our variety is isomorphic to the variety of pairs $(S,Y)$ satisfying \eqref{SY} and of the form  
$$S=\begin{pmatrix} 0 & x & 0 & 0\\ 1 & 0 & 0 & 0\\ 0 & 1 & \mu_1+\mu_2 & 0\\ 0 & 1 & 0 & \mu_1-\mu_2\end{pmatrix}, $$
$$Y=\begin{pmatrix} a & b & 0 & 0\\ c & -a & 0 & 0\\ y & z & (\mu_1-\mu_2)/2 & 0\\
u & v & 0 & (\mu_1+\mu_2)/2\end{pmatrix}.$$
As discussed in the previous section, $(S,Y)$ satisfies \eqref{SY} if and only if the $3\times 3$ minor matrices obtained by removing the 3rd (resp. the 4th) row and the 3rd (resp. the 4th) column satisfy these equations, i.e. belong to $D_{1,1}(\mu_1,-\mu_2)$ (resp.  $D_{1,1}(\mu_1,+\mu_2)$). The equation $SY+YS=\tau$ gives equations \eqref{c}-\eqref{a} and the following two equations:
\begin{equation}c+(\mu_1-\mu_2)u+v=0,
\label{c2}\end{equation}
\begin{equation} -a +(\mu_1-\mu_2)v +ux +(\mu_1+\mu_2)/2=0.\label{a2}\end{equation}
The equation $(Y-\frac{1}{2}S)^2=\mu_2^2$ for the two minor matrices is equivalent to equation \eqref{D1}, i.e.:
\begin{equation}0=y^2x-z^2+\frac{1}{4} +(\mu_1-\mu_2)y= \left(x-(\mu_1-\mu_2)^2\right)y^2 + \left((\mu_1-\mu_2) y +\frac{1}{2}\right)^2-z^2,                               \label{yz}
\end{equation}
\begin{equation}0=u^2x-v^2+\frac{1}{4} +(\mu_1+\mu_2)u= \left(x-(\mu_1+\mu_2)^2\right)u^2 + \left((\mu_1+\mu_2) u +\frac{1}{2}\right)^2-v^2.                               \label{uv}
\end{equation}
We can combine \eqref{c2}-\eqref{a2} with \eqref{c}-\eqref{a} to obtain two equations for $u,v,y,z$,  which we row reduce to 
\begin{equation} z-v+(\mu_1+\mu_2)y -(\mu_1-\mu_2)u=0,\label{1}\end{equation}
\begin{equation} 2\mu_2v+(x-(\mu_1+\mu_2)^2)y+(-x+\mu_1^2-\mu_2^2)u=\mu_2.\label{2}\end{equation}
Thus both $z$ and $v$ are functions of $y$ and $u$, and the solutions are:
\begin{equation*} z=
\frac{1}{2\mu_2}\left((\mu_1-\mu_2)^2-x\right)(y-u)+(\mu_1-\mu_2)y+\frac{1}{2},
\end{equation*}
\begin{equation*}
v=\frac{1}{2\mu_2}\left((\mu_1+\mu_2)^2-x\right)(y-u)+(\mu_1+\mu_2)u+\frac{1}{2},
\end{equation*}
Write $\alpha_{\pm}=\mu_1\pm \mu_2$ and introduce a new variable $w=(y-u)/2\mu_2$. The last two equations become then $z=(\alpha_-^2-x)w+\alpha_- y+\frac{1}{2}$ and $v=(\alpha_+^2-x)w+\alpha_+ u+\frac{1}{2}$.
Substituting into \eqref{yz}-\eqref{uv} we obtain 
$$ \bigl(x-\alpha_-^2\bigr)\bigl((y+\alpha_- w)^2+w-xw^2)\bigr)=0,$$
$$ \bigl(x-\alpha_+^2\bigr)\bigl((u+\alpha_+ w)^2+w-xw^2)\bigr)=0.$$
We easily check that
\begin{equation} y+\alpha_-w=u+\alpha_+ w=\frac{(\mu_1+\mu_2)y-(\mu_1-\mu_2)u}{2\mu_2}\label{t}\end{equation}
and, therefore, the defining ideal of $D_{0,1}(\mu_1,\mu_2)$ is generated by the single polynomial $(y+\alpha_- w)^2+w-xw^2$.
 Making the final substitution $t=y+\alpha_- w=u+\alpha_+ w$ we obtain the equation of the Atiyah-Hitchin manifold (also known as the $D_0$-manifold):
\begin{equation} t^2-xw^2+w=0.\label{D0}\end{equation}


\begin{remark} If $n=2$ or $n=3$, then we can extend the above identifications of complex structures to the case when one or both of $\mu_1,\mu_2$ is equal to zero. Indeed,  this means  replacing a semisimple orbit with a minimal nilpotent orbit. Since the elements of the minimal nilpotent orbit still satisfy the quadratic equation $A^2=0$, all the considerations from this and from the previous section remain valid. Observe that for $n=3$ we get a smooth manifold for any value of $\mu_1,\mu_2$, since, if $A+B$ is a regular matrix, then neither $A$ nor $B$ can be the singular point $0$ of the nilpotent variety.
\par
If $n=4$, the above identification also extends to the case $\mu_1=0$ or $\mu_2=0$, i.e. when one or both orbits become nilpotent. In this case, however, the canonical form \eqref{D0m} is no longer valid: the lower-right $2\times 2$-blocks of $S$ and $Y$ become nilpotent matrices. One needs to repeat the above computation separately for $S,Y$ of this form.
\end{remark}

\section{The case $m=1$ - the metrics\label{metric}}

As discussed in section \ref{slices} the manifolds $D_{k,m}(\mu_1,\mu_2)$ carry a natural hyperk\"ahler structure arising from their construction as a moduli space of solutions to Nahm's equations. In the case of $k=1$ there exist of course well-known hyperk\"ahler metrics on the complex surfaces found in the previous section: the ALF gravitational instantons \cite{Hi,AH}. Since our description of these spaces is very different from previously known ones, we are going to show that the metrics on $D_{k,1}$, $k=2,1,0$,  are the standard ones. 
\par
The hyperk\"ahler metric on $D_{k,m}(\mu_1,\mu_2)$ is obtained as a hyperk\"ahler quotient of the manifold $N_{n}$ ($n=2m+2-k$), described in \S\ref{slices}, and a pair of semisimple adjoint orbits with their Kronheimer-Biquard-Kovalev metrics. Both adjoint orbits admit a family of hyperk\"ahler structures parameterised by real numbers:  for any  $r_1\in \oR$ there is a hyperk\"ahler structures on $\sO_1$ such that the generic complex structure $I_\zeta$ ($\zeta\in \oP^1$) is that the adjoint orbit with eigenvalues $\pm (\mu_1+2r_1\zeta-\ol{\mu}_1\zeta^2)$ with multiplicities $k_1$ and $l_1$, and analogously for $\sO_2$. 
\par
The corresponding hyperk\"ahler metric on $D_{k,m}(\mu_1,\mu_2)$ can be found by identifying real sections of the twistor space which is obtained as the fibrewise complex-symplectic quotient of the product of fibred product of the twistor space of $N_n$ and the twistor spaces of the orbits $\sO_1$, $\sO_2$. The twistor space of $N_n$ was described at the end of section \ref{slices}. The sections of the twistor space of the orbits are simply $n\times n$ matrices with $\sO(2)$-entries (belonging to the orbit of $(\mu_i+2r_i\zeta-\ol{\mu}_i\zeta^2)I_{k_i}\oplus -(\mu_1+2r_i\zeta-\ol{\mu}_i\zeta^2)I_{l_i}$ for each $\zeta$, $i=1,2$). 
If $X_1\in \sO_1$, $X_2\in \sO_2$ and $(S,g)\in N_n$, then the complex moment map equation is $X_1+X_2=gSg^{-1}$. In terms of our original equation  $A+B=S$, $A+B=g^{-1}(X_1+X_2)g$, i.e. $Y=\frac{1}{2}g^{-1}(X_1-X_2)g$. It follows that
\begin{equation}\tilde S= D(\zeta) (S/\zeta^2)D(\zeta)^{-1},\quad \tilde Y=D(\zeta)\exp(S/\zeta)(Y/\zeta^2)\exp(-S/\zeta)D(\zeta)^{-1}.\label{transition}\end{equation}

We shall now identify the twistor lines for our surfaces $D_{k,1}(\mu_1,\mu_2)$, $k=2,1,0$. We begin with the special case $D_{2,1}(\mu/2,\mu/2)=D_{2,1}(\mu/2,-\mu/2)$. 
\par
{\boldmath${D_{2,1}(\mu/2,\mu/2)}$.}
As long as $x\neq 0$, we can diagonalise $S$ to $\diag(\lambda,-\lambda)$. The corresponding $Y_d$ (which anticommutes with $S$) is of the form $\begin{pmatrix} 0 & u\\ v & 0\end{pmatrix}$. The relationship between $u,v$ and $a,c$ in \S\ref{Ex} is
$$ a=\frac{u+v}{2}, \enskip c=\frac{u-v}{2\lambda}.$$
If $S$ is diagonal, the transition functions in \eqref{twistor} become $\tilde S=S/\zeta^2$, $\tilde g=g\exp(-S/\zeta)$, and the ones in \eqref{transition} are $\tilde S=S/\zeta^2$, $\tilde Y=\exp(S/\zeta)(Y/\zeta^2)\exp(-S/\zeta)$. Therefore the transition functions for $\lambda,u,v$  are $\tilde{\lambda}=\frac{\lambda}{\zeta^2}$, $\tilde{u}=e^{2\lambda/\zeta}\frac{u}{\zeta^2}$, $\tilde{v}=e^{-2\lambda/\zeta}\frac{v}{\zeta^2}$, and it follows that in terms of $\tilde{a},\tilde{c}$ we have:
$$ \tilde{a}=\frac{\tilde u+\tilde v}{2}=\frac{1}{\zeta^2}\left(a\cosh\frac{2\lambda}{\zeta}+c \lambda \sinh \frac{2\lambda}{\zeta}\right),$$
$$ \tilde{c}=\frac{\tilde u-\tilde v}{2\tilde \lambda}=\frac{a}{\lambda}\sinh\frac{2\lambda}{\zeta}+c \cosh \frac{2\lambda}{\zeta}.$$
In particular observe that $a-\lambda c$ is a section of $L^{2}(2)$, where $L^{2}$ is a line bundle over $|\sO(2)|\simeq T\oP^1$ with transition function $\exp(-2\lambda/\zeta)$ from $\zeta\neq \infty$ to $\zeta\neq 0$. The equation \eqref{D2} can be written as $(a+\lambda c)(a-\lambda c)=\frac{1}{4}(\mu^2-x)$, so that twistor sections are given by a section $x(\zeta)$ of $\sO(4)$  and a section $s(\zeta,\lambda)$ of $L^2(2)$  over the elliptic curve $\lambda^2=x(\zeta)$ satisfying $s(\zeta,\lambda)s(\zeta,-\lambda)=\frac{1}{4}(\mu^2(\zeta)-x(\zeta))$ (satisfying the natural reality conditions). With a bit of extra care (cf. \cite[p.308-309]{BieSU}) one can show that $2$ zeros of $s(\zeta,\lambda)$ occur  at the intersection of the the elliptic curve with $\lambda=\mu(\zeta)$ and the two other zeros  at the intersection of the the elliptic curve with $\lambda=-\mu(\zeta)$.

\medskip

{\boldmath${D_{2,1}(\mu_1,\mu_2)}$, $\mu_1\neq\pm \mu_2$.} 
 Following Hitchin \cite{Hi}, let us first rewrite \eqref{D2}. After multiplying  it by $x$ (this will not affect the determination of twistor lines), we can rewrite it as
$$ xa^2-(xc-\tau/2)^2+\frac{1}{4}\bigl(x-(\mu_1-\mu_2)^2\bigr)\bigl(x-(\mu_1+\mu_2)^2\bigr)=0,$$
where $\tau=\mu_1^2-\mu_2^2$.
After introducing a new variable $w=xc-\tau/2$ we obtain the equation $w^2-xa^2=\frac{1}{4}\bigl(x-(\mu_1-\mu_2)^2\bigr)\bigl(x-(\mu_1+\mu_2)^2\bigr)$. 
\par
We now proceed as in the case $\mu_1=\pm\mu_2$. 
This time a point $(S,Y)$ of our variety satisfies $SY+YS=\tau$ with $\tau\neq 0$. If $x\neq 0$, we can write $Y=\frac{1}{2\tau}S^{-1}+Y^\prime$ with $Y^{\prime}$ anticommuting with $S$ as before. Thus after diagonalising $S$, $Y$ becomes 
$$\begin{pmatrix} \frac{1}{2\lambda\tau} & u\\ v & \frac{-1}{2\lambda\tau}\end{pmatrix},$$
and, consequently,
$$ a=\frac{u+v}{2}, \enskip c=\frac{u-v}{2\lambda} +\frac{1}{2\lambda^2\tau},$$
so that 
$$ a=\frac{u+v}{2}, \enskip w=\lambda\frac{u-v}{2}.$$
The same computation as above shows that the transition functions for $a$ and $w$ are 
$$ \tilde{a}=\frac{\tilde u+\tilde v}{2}=\frac{1}{\zeta^2}\left(a\cosh\frac{2\lambda}{\zeta}+w \lambda \sinh \frac{2\lambda}{\zeta}\right),$$
$$ \tilde{w}=\tilde \lambda\frac{\tilde u-\tilde v}{2}={a}{\lambda}\sinh\frac{2\lambda}{\zeta}+w \lambda^2\cosh \frac{2\lambda}{\zeta}.$$
Thus $w-\lambda a$ is a section $s(\zeta,\lambda)$ of $L^{2}(4)$ over the elliptic curve $\lambda^2=x(\zeta)$ satisfying $s(\zeta,\lambda)s(\zeta,-\lambda)=\frac{1}{4}\bigl(x(\zeta)-(\mu_1(\zeta)-\mu_2(\zeta))^2\bigr)\bigl(x(\zeta)-(\mu_1(\zeta)+\mu_2(\zeta))^2\bigr)$ (again satisfying the natural reality conditions). Again, the precise location of the zeros of $s(\zeta,\lambda)$ can be determined.
This is Hitchin's description of the hyperk\"ahler metric on the deformation of the $D_2$-singularity \cite{Hi}.

\medskip

{\boldmath${D_{1,1}(\mu/2,\mu/2)}$.} We proceed similarly. In the basis in which $S=\diag(\lambda,-\lambda,\mu)$, $Y$ is of the form 
$$\begin{pmatrix} 0 & u & 0\\ v & 0 &0 \\ 0 & 0 &0\end{pmatrix}.
$$
The transition matrices between this basis and the one in which $S$ and $Y$ have the form \eqref{SY3} are easily computed using corresponding bases of $\cx[z]/((z^2-\lambda^2)(z-\mu))$. They are:
$$ V=\begin{pmatrix} 1 & \lambda & 0\\ 1 &-\lambda & 0\\ 1 & \mu & \mu^2-\lambda^2\end{pmatrix},\quad V^{-1}=\begin{pmatrix} \frac{1}{2} & \frac{1}{2} & 0\\[0.5em] \frac{1}{2\lambda} & -\frac{1}{2\lambda} & 0\\[0.5em] \frac{1}{2\lambda(\lambda-\mu)} &  \frac{1}{2\lambda(\lambda+\mu)} & \frac{1}{\mu^2-\lambda^2}\end{pmatrix}.$$
We have:
$$\begin{pmatrix} a & b & 0\\ c & -a & 0\\ y & z & 0\end{pmatrix}=V^{-1}\begin{pmatrix} 0 & u & 0\\ v & 0 &0 \\ 0 & 0 &0\end{pmatrix} V,$$
and consequently:
\begin{equation}y=\frac{v}{2\lambda(\lambda+\mu)}+\frac{u}{2\lambda(\lambda-\mu)},\quad z=\frac{v}{2(\lambda+\mu)}-\frac{u}{2(\lambda-\mu)}.\label{yz2}
\end{equation}
The transition functions for $\lambda, u,v$ are the same as in the $D_2$-case and from an analogous calculation we obtain:
$$ \tilde{y}= \zeta^2 \left(y \cosh \frac{2\lambda}{\zeta}-\frac{z}{\lambda}\sinh\frac{2\lambda}{\zeta}\right),\quad
 \tilde{z}= z\cosh \frac{2\lambda}{\zeta}-y{\lambda}\sinh\frac{2\lambda}{\zeta}.$$
 In particular, observe that $z+\lambda y$ is a section of $L^2$. The equation  of our $D_1$-manifold can be written as $(z+\lambda y)(z-\lambda y)=\frac{1}{4}$, so that twistor sections are given by a section $x(\zeta)$ of $\sO(4)$  and a nonvanishing section $s(\zeta,\lambda)$ of $L^2$  over the elliptic curve $\lambda^2=x(\zeta)$ satisfying $s(\zeta,\lambda)s(\zeta,-\lambda)=\frac{1}{4}$. This is the twistor description of the Atiyah-Hitchin metric on the $D_1$-manifold (it is easy to see that the reality conditions are the same; one can also check that the symplectic forms coincide, but already the above information  
determines the hypercomplex structure, hence the Levi-Civita connection, hence the metric up to a constant factor).

\medskip

{\boldmath${D_{1,1}(\mu_1,\mu_2)}$, $\mu_1\neq\mu_2$.} As in the $D_2$-case we multiply \eqref{D1} by $x$ and rewrite it as 
\begin{equation} \bigl(yx+(\mu_1-\mu_2)/2\bigr)^2-xz^2=\frac{1}{4}(\mu_1-\mu_2)^2.\label{D1new}\end{equation}
With a new variable $q=yx+(\mu_1-\mu_2)/2$  this becomes $q^2-xz^2=\frac{1}{4}(\mu_1-\mu_2)^2$. Proceeding as in the $D_2$-case we recover Chalmers' description \cite{Cha} of twistor lines for Dancer's deformation of the Atiyah-Hitchin metric.

\medskip

{\boldmath${D_{0,1}(\mu_1,\mu_2)}$.} We first discuss the twistor description of the Atiyah-Hitchin metric. The Atiyah-Hitchin manifold  arises as the quotient of the $D_1$-manifold $z^2-y^2x=1/4$ by the involution $(z,y)\mapsto (-z,-y)$. Setting $s=z^2$, $t=4yz$, $w=-4y^2$, and substituting $s=(1-wx)/4$ into $t^2=-4ws$, we obtain $t^2-xw^2+w=0$. Multiply this last equation  by $w$ and rewrite it as $(xw-1/2)^2-xt^2=1/4$, or $(xw-1/2 +t\sqrt{x})(xw-1/2 -t\sqrt{x})=1/4$. Now observe that 
$$ xw-1/2 \pm t\sqrt{x}=-2(z\pm y\sqrt{x})^2.$$
The description of the twistor space of $D_{1,1}(\mu,\mu)$ implies that $z\pm y\sqrt{x}$ is a nonvanishing section section of $L^{\pm 2}$, so that $xw-1/2 \pm t\sqrt{x}$ becomes a nonvanishing section of $L^{\pm 4}$. We shall therefore show that the hyperk\"ahler metric on ${D_{0,1}(\mu_1,\mu_2)}$ is the Atiyah-Hitchin metric, if we can show that $t$ and $w$ arising via the calculation in \S\ref{Ex} do make $xw-1/2 \pm t\sqrt{x}$ a holomorphic section of $L^{\pm 4}$ (or any $L^{\pm c}$, $c>0$, which corresponds to rescaling the metric). 
\par
According to \S\ref{Ex}, restricting the $4\times 4$ matrices $S,Y$ to appropriate $3\times 3$ minor matrices produces two $D_1$-manifolds with equations
$$ y^2x-z^2+\frac{1}{4} +(\mu_1-\mu_2)y=0, \quad u^2x-v^2+\frac{1}{4} +(\mu_1+\mu_2)u=0.$$
We can rewrite these as in \eqref{D1new}, i.e.:
$$ q^2-xz^2=\frac{1}{4}(\mu_1-\mu_2)^2, \quad p^2-xv^2=\frac{1}{4}(\mu_1+\mu_2)^2$$
where $q=yx+(\mu_1-\mu_2)/2$ and $p=ux+(\mu_1+\mu_2)/2$. It follows from the twistor description of ${D_{1,1}(\mu_1,\mu_2)}$ that $q\pm z\sqrt{x}$ and $p\pm v\sqrt{x}$ are sections of $L^{\pm 2}(2)$. Using the definition of $t$ given in \eqref{t} and  $w=(y-u)/2\mu_2$, together with \eqref{1}, we obtain 
$$ \frac{q+z\sqrt{x}-p-v\sqrt{x}}{2\mu_2}=\frac{(q-p)+(z-v)\sqrt{x}}{2\mu_2}={xw-1/2-t\sqrt{x}}$$
and
$$ \frac{q-z\sqrt{x}-p+v\sqrt{x}}{2\mu_2}=\frac{(q-p)-(z-v)\sqrt{x}}{2\mu_2}={xw-1/2+t\sqrt{x}}.$$
Thus both $xw-1/2+t\sqrt{x}$ and $xw-1/2-t\sqrt{x}$ are meromorphic sections of $L^2$ and $L^{-2}$, respectively, with poles possible only at zeros of $\mu_2(\zeta)$. Equation \eqref{D0} and the calculation at the beginning of this subsection imply, however, that their product is constant (equal to $1/4$). Therefore $xw-1/2\pm t\sqrt{x}$ are holomorphic sections of $L^{\pm 2}$ and the metric on ${D_{0,1}(\mu_1,\mu_2)}$  is the Atiyah-Hitchin metric. 

\begin{remark} Due to general facts about isometries between real-analytic Riemannian manifolds, the natural maps $\phi: D_{1,m}(\mu_1,\mu_2)\to D_{2,m}(\mu_1,\mu_2)$ and $\phi_1:D_{0,m}(\mu_1,\mu_2)\to D_{1,m}(\mu_1,\mu_2)$, described in \S\ref{nonvan}, cannot be (even local) isometries. On the other hand, it appears that the first of these maps is well-defined on each fibre of the twistor space and maps real sections to real sections. To see that this is not so, observe that computing $v$ from \eqref{yz2} yields $v=(\lambda+\mu)(z+\lambda y)$ (and similarly $u=(\mu-\lambda)(z-\lambda y)$), so that the map $\phi$ sends a real section of the twistor space of $D_{1,1}(\mu/2,\mu/2)$ to a pair of sections $s_{\pm}$ of $L^{\pm 2}(2)$ over the  elliptic curve $\lambda^2=x(\zeta)$ such that all zeros of $s_+$ (resp. $s_-$) occur at the intersection points with $\lambda=-\mu(\zeta)$ (resp. $\lambda=\mu(\zeta)$). This  real section does {\em not} arise from the hyperk\"ahler structure of $D_{2,1}(\mu/2,\mu/2)$ (see 
the remark at the end of the calculation of the metric for $D_{2,1}(\mu/2,\mu/2)$)).
\par
The second of these maps, $\phi_1:D_{0,m}(\mu_1,\mu_2)\to D_{1,m}(\mu_1,\mu_2)$, is well-defined only if $\mu_2\neq 0$, which means that is not well-defined on all fibres of the twistor space.
\end{remark}

\section{Hilbert schemes of points transverse to a projection\label{Hilbpi}}

Let $X$ be a complex manifold, $C$ a $1$-dimensional complex manifold, and $\pi:X\to C$ a surjective holomorphic map. Following Atiyah and Hitchin \cite{AH} we define an open subset $X_\pi^{[n]}$ of the Hilbert scheme $X^{[n]}$  of $n$ points in $X$ as consisting of these $0$-dimensional subschemes $Z$ of length $n$ for which $\pi_\ast \sO_Z$ is a cyclic $\sO_C$ sheaf.
Equivalently, $\pi:Z\to \pi(Z)$ is an isomorphism onto its scheme-theoretic image.
\par
Let $Z\in X_\pi^{[n]}$ and suppose that  $\pi:Z\to \pi(Z)$ is an isomorphism, i.e. $\pi_\ast \sO_Z=\sO_{\pi(Z)}$.
 If $t_0$ is a point in $\pi(Z)$ and $t$ is a local coordinate on $C$, then a  neighbourhood of $t_0$ is of the form $\cx[t]/\bigl(t^m\bigr)$ for some $m\leq n$. Since $\pi: Z\to \pi(Z)$ is an isomorphism, there exists a morphism $\phi: \cx[t]/\bigl(t^m\bigr)\to X$, the image of which is the corresponding open subset of $Z$. Such a morphism $\phi$ is an equivalence class of local smooth sections of $\pi$ truncated up to order $m$. Globally, $Z$ is an equivalence class of local smooth sections of $\pi$, defined in a neighbourhood of $\pi(Z)$, where the equivalence relation is as above. In other words, if $s$ is a local section on $U\subset \cx$, then the defining ideal of $Z$ is $I_s+\pi^\ast I_{\pi(Z)}$, where $I_s$ is the defining ideal of $s(U)$. We can formulate this as follows:
\begin{proposition}(cf. \cite[Ch.6]{AH}) $X_\pi^{[n]}$ parameterises $0$-dimensional subschemes $Z$ of $X$ of length $n$ such that $I_Z=I_s+\pi^\ast I_T$, where $T\in C^{[n]}\simeq S^n C$, $s$ is a local section of $\pi$ defined in a neigbourhood $U$ of $T$, and $I_s$ is the defining ideal of $s(U)$.
\end{proposition}
\begin{remark} It follows from the construction that $\pi$ induces a surjective holomorphic map $\pi^{[n]}:X_\pi^{[n]}\to C^{[n]}\simeq S^n(C)$. 
\end{remark}

We shall call $X_\pi^{[n]}$ {\em Hilbert scheme of $n$ points transverse to $\pi$} or simply {\em transverse Hilbert scheme}.

\bigskip

Suppose now that $X\subset \cx^{k}$ is an affine variety  and $\pi:X\to \cx$ is a restriction of a polynomial $P:\cx^k\to \cx$ to $X$. Without loss of generality we may assume that $P(w_1,\dots,w_k)=w_k$. Indeed, if this is not the case, then we can view $X$ as an affine variety in $\cx^{k+1}$ by adding the equation $w_{k+1}=P(w_1,\dots,w_k)$, so that $\pi$ becomes the projection onto the last coordinate. 
\par
Let us therefore assume that $\pi(w_1,\dots,w_k)=w_k$ and write $z=w_k$.
 The ideal of a $T\in \cx^{[n]}\simeq S^n\cx$ is generated by a monic polynomial $q(z)\in \cx[z]$ of degree $n$. 
If $Z\in X_\pi^{[n]}$ and $\pi(Z)=T$, then $q(z)\in I_Z$. If $q(z)=\prod_{i=1}^s (z-z_i)^{m_i}$, then  a local section of $\pi$ in a neighbourhood of $z_i$ modulo $q(z)$ is of the form $(w_1^i(z),\dots,w_{k-1}^i(z),z)$, where $w_j^i(z)$ is polynomial of degree $\leq m_i-1$. These local polynomials can be combined, using Lagrange interpolation, to give polynomials $w_1(z),\dots,w_{k-1}(z)$ of degree at most $n-1$ satisfying $w_j(z)=w_j^i(z)\mod (z-z_i)^{m_i}$. 
It follows that the defining ideal of a point $Z$ in $X_\pi^{[n]}$ is given by 
$$ \bigl(q(z), f(w_1(z),\dots,w_{k-1}(z),z)\bigr)_{f\in I},$$
where $q(z)$ is a monic polynomial of degree $n$, $I$ is the defining ideal of $X$, and $w_1(z),\dots, w_{k-1}(z)$  are polynomials of degree at most $n-1$. Thus $X_\pi^{[n]}$ is an affine variety in $\cx^{kn}$, with coordinates $p_{ij},q_j$, $1\leq i\leq k-1$,  $0\leq j\leq n-1$ defined by equations
\begin{equation} \enskip f\left(\sum_{j=0}^{n-1}p_{1j}z^j,\dots, \sum_{j=0}^{n-1}p_{n-1,j}z^j,z\right)=0 \mod \left(z^n-\sum_{j=0}^{n-1}q_jz^j\right)\label{aff-Hilb}\end{equation}
for every $f\in I$.

\begin{example}
Let $X=\cx^\ast\times \cx$ and $\pi$ the projection onto the second coordinate. We can view $X$ as the affine variety $\{(x,y,z)\in \cx^3;xy=1\}$ with $\pi(x,y,z)=z$. 
According to the above description $X^{[n]}_\pi$ is the variety of triples of polynomials $x(z),y(z),q(z)$ of degrees $n-1,n-1$ and $n$ and $q$ monic such that $x(z)y(z)=1\mod q(z)$. In other words $x(z)$ (or $y(z)$) does not vanish at any of the roots of $q(z)$ and, consequently, $X^{[n]}_\pi$ is isomorphic to the space of based rational maps of degree $n$ (cf. \cite{AH}).\label{ratio} 
\end{example}

\begin{example} Let $X$ be the double cover of the Atiyah-Hitchin manifold, i.e. an affine surface in $\oR^3$ defined by the equation $x^2-zy^2=1$. Again, $X^{[n]}_\pi$ is the variety of triples of polynomials $x(z),y(z),q(z)$ of degrees $n-1,n-1$ and $n$ and $q$ monic, such that $x(z),y(z),z$ satisfy the defining equation modulo $q(z)$. Alternatively, consider the quadratic extension $z=u^2$, so that the defining equation becomes $(x+uy)(x-uy)=1$. If $x(z)$ and $y(z)$ are polynomials of degree $n-1$, then $x(z)\pm uy(z)=x(u^2)\pm uy(u^2)$ and $q(z)=q(u^2)$. In other words, $q(u^2)$ is a polynomial of degree $2n$ with all coefficients of odd powers equal to $0$ and $p(u)=x(u^2)+  uy(u^2)$ is a polynomial of degree $2n-1$ satisfying $p(u)p(-u)=1 \mod q(u^2)$.   Thus $X^{[n]}_\pi$ is the space of degree $2n$ based rational maps of the form $p(u)/q(u^2)$  with $p$ satisfying the above condition. 
\label{AH-Hilb}\end{example}

\subsection{Hyperk\"ahler metrics\label{fibre}}
As observed by Atiyah and Hitchin \cite{AH}, the definition of $X_\pi^{[n]}$ is well-suited to twistorial constructions of new hyperk\"ahler metrics (or hypercomplex structures) from old ones. Namely, let $Z$ be the twistor space of a hypercomplex or hyperk\"ahler manifold. In particular, $Z$ comes equipped with a holomorphic submersion $p:Z\to \oP^1$ and an antiholomorphic involution (real structure) $\sigma$ covering the antipodal map. Suppose that $Z$ also admits a holomorphic map $\pi$ to the total space of the line bundle $\sO(2r)$ which preserves the fibres over $\oP^1$ and is compatible with the real structures, where the canonical real structure of $\sO(2r)$ is
$$\left(\zeta,\eta\Bigl(\frac{\partial}{\partial \zeta}\Bigr) ^{\otimes r}\right) \mapsto \left(-1/{\bar\zeta},(-1)^r\frac{\bar\eta}{\bar{\zeta}^{2r}}\Bigl(\frac{\partial}{\partial \zeta}\Bigr) ^{\otimes r}\right).$$
Assuming that $\pi$ is surjective, we can apply the transverse Hilbert scheme construction fibrewise and obtain a new twistor space $Z^{[n]}_\pi$, which also fibres over $\oP^1$ and has an induced real structure. Moreover, if $\dim Z=2$ and $Z$ had a fibrewise $\sO(2)$-twisted symplectic form, then so does $Z^{[n]}_\pi$ \cite{Beau}. 

From the description of $X_\pi^{[n]}$, given above, it is clear that a  section $s$ of $p:Z^{[n]}_\pi\to \oP^1$ corresponds to a degree $n$ curve $C$ in $Z$ such that $\pi_{|C}:C\to \pi(C)$ is an isomorphism, or, equivalently, a curve of degree $n$ in the total space of the line bundle $\sO(2r)$ together with a lift to $Z$. Moreover, the results of \cite{Bie2} imply that the normal bundle of $s(\oP^1)$ is the sum of $\sO(1)$-s if and only if the normal bundle $N$ of $C$ in $Z$ satisfies the condition $H^0(C,N(-2))=H^1(C,N(-2))=0$. 

\begin{example} Applying this construction to the twistor space of $X=\cx^\ast\times \cx$ (cf. Example \ref{ratio}) produces the natural complete hyperk\"ahler metric on the moduli space $\sM_n$ of framed euclidean $SU(2)$-monopoles of charge $n$ \cite{AH}.  On the other hand applying the construction to the double cover of the Atiyah-Hitchin manifold as in Example \ref{AH-Hilb} produces a totally geodesic submanifold  of $\sM_{2n}$, consisting of monopoles which are symmetric about the origin (in particular centred)  and have total phase equal to $1$. This follows easily by considering the effect of the involution $x\to -x$ in $\oR^3$  on the twistor space of monopoles (cf. \cite{HMM}).  
\label{tot}
\end{example}

\section{Example: hyperk\"ahler metrics on $\bigl(\cx^2\bigr)^{[n]}_{\pi}$\label{C2pi}}

We consider $X=\cx^2$ and $\pi(x,y)=xy$. As described in the previous section, we view $X$ as the variety
$$ \{(x,y,z)\in \cx^3;\; xy=z\},$$
and $\pi(x,y,z)=z$. The variety $\bigl(\cx^2\bigr)^{[n]}_\pi$ is then a $2n$-dimensional affine variety in $\cx^{3n}$ defined by equations
\begin{equation} \left(\sum_{j=0}^{n-1}x_jz^j\right)\left(\sum_{j=0}^{n-1}y_jz^j\right)=z \mod \left(z^n-\sum _{j=0}^{n-1}q_jz^j\right).\label{C2}\end{equation}
For example, the defining equations of $\bigl(\cx^2\bigr)_\pi^{[2]}$ are
$$ x_0y_0+x_1y_1q_0=0,
$$
$$x_1y_0+x_0y_1+x_1y_1q_1=1.
$$

We recall Nakajima's description of the Hilbert scheme of $n$ points in $\cx^2$ \cite{Nak}: $\bigl(\cx^2\bigr)^{[n]}$ is isomorphic to the the quotient by $GL(n,\cx)$ of triples $(B_1,B_2,i)$ consisting of commuting $n\times n$ matrices $B_1,B_2$ and an $i\in \cx^n$ which satisfy the following stability condition: there is no proper subspace $V$ of $\cx^n$ containing $i$ and invariant under $B_1$ and $B_2$. The defining ideal of the corresponding subscheme $Z$  given by
$$ I_Z=\{f\in\cx[x,y];\; f(B_1,B_2)=0\}.$$
The defining ideal of $\pi(Z)$ consists of functions $g\in \cx[z]$ such that $\pi^\ast g=0$, i.e. $g(B_1B_2)=0$. We have an embedding $\sO_{\pi(Z)}\hookrightarrow \pi_\ast\sO_{Z}$, $g\mapsto g(xy)$, which is an isomorphism exactly then, when $\dim \sO_{\pi(Z)}=n$, which means that $B_1B_2$ is a regular matrix.
\par
We conclude therefore that  $\bigl(\cx^2\bigr)^{[n]}_\pi$ consists of $GL(n,\cx)$-orbits of triples $(B_1,B_2,i)$ as above and such that $B_1B_2$ is a regular matrix. Using again the isomorphism between $Z$ and $\pi(Z)$ we also conclude that $i$ is a cyclic vector for $B_1B_2$.  Every conjugacy class of regular matrices contains a unique companion matrix $S$, i.e. a matrix of the form \eqref{S}.
If $B_1B_2=S$, then we can conjugate $B_1$ and $B_2$ by an element of the centraliser of $S$ in order to make the vector $i$ equal to $e_1=(1,0,\dots,0)^T$. The pair $(S,e_1)$ has trivial stabiliser and we have thus shown:
\begin{proposition} The variety $\bigl(\cx^2\bigr)^{[n]}_\pi$ is isomorphic to the variety of triples $(B_1,B_2,S)$ of $n\times n$ matrices, such that $S$ is of the form \eqref{S}, $B_1$,$B_2$ commute, and $B_1B_2=S$.\hfill $\Box$
\label{prop}\end{proposition}
To recover the description of $\bigl(\cx^2\bigr)^{[n]}_\pi$ given in \eqref{C2}, we observe that both $B_1$ and $B_2$ commute with the regular matrix $S$, i.e. we can uniquely write:
$$ B_1=\sum_{i=0}^{n-1} x_iS^i,\quad B_2=\sum_{i=0}^{n-1} y_iS^i,$$
and since $S^k=\sum_{i=0}^{n-1} q_i S^i$, the equation $B_1B_2=S$ is equivalent to \eqref{C2}.

$\cx^2$ carries at least two complete hyperk\"ahler metrics: the flat one and the Taub-NUT. We can therefore apply the fibrewise construction described in section \ref{fibre} to both of them and try to obtain new hyperk\"ahler metrics in higher dimensions.
\par
The twistor space $Z$ of the flat metric on $\oR^4$ is the total space of $\sO(1)\oplus \sO(1)$ over $\oP^1$. The map $\pi(xy)=xy$ induces a projection to $\sO(2)$ given by fibrewise multiplication. As discussed above, a section of $p:Z^{[k]}_\pi\to \oP^1$ would correspond to a degree $n$ curve in $|\sO(2)|$ which can be lifted to $|\sO(1)\oplus \sO(1)|\simeq \oP^3\backslash \oP^1$. A degree $n$ curve in $|\sO(2)|$ has genus $(n-1)^2$ and its lift would be a degree $n$ curve in $\oP^3$ having such a genus. There are, however, no such curves owing to a classification result of Hartshorne \cite{Hart}. Thus we conclude that $p:Z^{[n]}_\pi\to \oP^1$ has  no sections and we do not obtain new hyperk\"ahler metrics.

\subsection{Taub-NUT}

The twistor space of the Taub-NUT metric is described for example in \cite[pp. 393--395]{Besse}. There is actually a  family of such metrics depending on a positive real parameter $c$ with the twistor space defined as
$$ Z_c=\{(x,y)\in L^{c}(1)\oplus L^{-c}(1); xy=z\},
$$
where $L^c$ is a line bundle over $|\sO(2)|\simeq T\oP^1$ with transition function $\exp(-c\eta/\zeta)$ from $\zeta\neq \infty$ to $\zeta\neq 0$, and $z$ is the tautological section of $\rho^\ast \sO(2)$ over $|\sO(2)|$, where $\rho: |\sO(2)|\to \oP^1$ is the projection.
\par
It follows that a section of $p:(Z_c)^{[n]}_\pi\to \oP^1$ corresponds to a degree $n$ curve $C$ in $|\sO(2)|$ which can be lifted to $Z_c$, i.e. to a pair of sections $s_1$ of $L^c(1)_{|C}$ and $s_2$ of $L^{-c}(1)_{|C}$ such that $(s_1)+(s_2)=(z)$. It is a priori unclear that such sections exist for any $n$ and, even if they do, that there is a $\sigma$-invariant family of them defining a complete hyperk\"ahler metric. We shall now show that this is so by giving a construction of a complete hyperk\"ahler metric on  $\bigl(\cx^2\bigr)^{[n]}_\pi$ as a hyperk\"ahler quotient.
\par
We observe namely that 
Proposition \ref{prop} is actually a description of $\bigl(\cx^2\bigr)^{[n]}_\pi$ as a complex-symplectic quotient and it suggests what the hyperk\"ahler quotient should be. We begin with the space $V$ of $n\times n$ quaternionic matrices and two copies of $N_n$ (defined in \S\ref{slices}), but with metric rescaled by $c$ (equivalently: we consider solutions to Nahm's equations on $(0,c]$ rather than on $(0,1]$). $V$ has two tri-hamiltonian $U(n)$ actions and $N_n$ also has a tri-Hamiltonian action of $U(n)$. The hyperk\"ahler quotient $M_n$ of $V\times N_n\times N_n$ by $U(n)\times U(n)$  is the hyperk\"ahler slice to $V$ as described in section \ref{slices}. It is therefore a smooth and complete hyperk\"ahler manifold. It remains to identify its complex structures.  The complex structure of $V$ is that of pairs $B_1,B_2$ of complex matrices  with the two $\g\fL(k,\cx)$-valued moment maps given by $B_1B_2$ and $B_2B_1$. The complex structure of $N_n$ is that of pairs $(S,g)$, $S$ having the form \eqref{S} 
and $g\in GL(n,\cx)$, with the moment map $gSg^{-
1}$. It follows from the results of \cite{Bie1}  (recalled in \S\ref{slices}) that $M_n$ is biholomorphic to the variety of quadruples $(B_1,B_2,S_1,S_2)$, with $S_1,S_2$ of the form \eqref{S} and satisfying $B_1B_2=S_1$, $B_2B_1=S_2$. Since the characteristic polynomials of $B_1B_2$ and $B_2B_1$ are the same, it follows that $S_1=S_2$ and $B_1,B_2$ commute. Thus $M_n$ as a complex manifold is isomorphic to $\bigl(\cx^2\bigr)^{[n]}_\pi$.
\par
We observe that in case $n=1$, $V=\oH$ and $N_1=S^1\times \oR^3$ so that the hyperk\"ahler quotient $M_1$ is the Taub-NUT metric on $\cx^2$. It is straightforward to check, given the description of the twistor space of $N_n$ in \S\ref{slices}, that the twistor space of $M_n$ is isomorphic to $(Z_c)_\pi^{[n]}$.

\begin{remark} Just as the Taub-NUT metric itself, the transverse Hilbert scheme of $n$ points on it also admits a tri-Hamiltonian circle action. In the case $n=2$ we can then perform the hyperk\"ahler quotient by this circle and obtain again a hyperk\"ahler $4$-manifold. It is easy to compute from equations \eqref{pols} that this is again a deformation of the $D_2$-singularity.  
\end{remark}

\section{Hyperk\"ahler metrics on $(D_k)^{[m]}_\pi$. \label{HilbDk}}

Recall the varieties $D_{k,m}(\mu_1,\mu_2)$ defined in section \ref{nonvan}. These are regular slices to sums of two orbits and carry, therefore, natural complete hyperk\"ahler metrics. In the case $m=1$ we have identified them as the $D_0$, $D_1$ and $D_2$ ALF gravitational instantons. We are now going to prove that  $D_{k,m}(\mu_1,\mu_2)$ is the Hilbert scheme $\bigl(D_{k,1}(\mu_1,\mu_2)\bigr)^{[m]}_\pi$ of $m$ points on $D_{k,1}(\mu_1,\mu_2)$ transverse to the projection onto the $x$-coordinate (in equations  
\eqref{D2}, \eqref{D1} and \eqref{D0}). 

\medskip

{\boldmath $D_{2,m}(\mu_1,\mu_2)$}.
Let $(S,Y)$ be an element of $D_{2,m}(\mu_1,\mu_2)$. The characteristic polynomial $P(z)$ of $S$ is, thanks to Proposition \ref{alg}, of the form $\prod_{i=1}^m (z^2-x_i)$. If the $x_i$ are distinct, we can conjugate $S$ to a block-diagonal form, with $2\times 2$ blocks 
\begin{equation} S_i= \begin{pmatrix} 0 & x_i\\ 1 & 0\end{pmatrix}.
\label{S_i}\end{equation}
Viewing $S$ as multiplication by $z$ on $\cx[z]/(P(z))$, this corresponds to a change of basis from $1,\dots, z^{2m-1}$ to 
\begin{equation} f_1(z),zf_1(z),f_2(z),zf_2(z),\dots, f_m(z),zf_m(z),\quad f_i(z)=\frac{\prod_{j\neq i} (z^2-x_j)}{\prod_{j\neq i}(x_i-x_j)}.
\label{f_i}\end{equation}
In this basis $Y$ becomes also block-diagonal with $2\times 2$ blocks $Y_i$. Indeed, the equation $SY+YS=\tau$, implies that any $2\times 2$ minor matrix $Y^{ij}$ of $Y$ of the form
$\begin{pmatrix} y_{2i-1,2j-1} & y_{2i-1,2j}\\ y_{2i,2j-1} & y_{2i,2j}\end{pmatrix}$ satisfies the equation $S_iY^{ij}+Y^{ij}S_j=\delta_{ij}\tau$. Since $x_i\neq x_j$ for $i\neq j$, $Y^{ij}=0$ for $i\neq j$. The diagonal blocks $(S_i,Y_i)$ belong to $D_{2}(\mu_1,\mu_2)$ and, consequently, the open subset $D_{2,m}(\mu_1,\mu_2)^o$ of $D_{2,m}(\mu_1,\mu_2)$, where the $x_i$ are distinct is isomorphic to 
\begin{equation} \left\{ (p_1,\dots,p_m)\in D_2(\mu_1,\mu_2);\; p_i=(a_i,c_i,x_i),\enskip \forall_{i\neq j} x_i\neq x_j\right\}/ \Sigma_m.
\end{equation}
To describe the locus  in $D_{2,m}(\mu_1,\mu_2)$ where some $x_i$ coalesce, we observe first that the change of basis matrix from $f_1,\dots,f_m,zf_1,\dots,zf_m$ to $ 1,\dots, z^{2m-1}$
is 
$$ \begin{pmatrix} V(x_1,\dots,x_m) & 0\\ 0 &  V(x_1,\dots,x_m)\end{pmatrix}
$$
where $V=V(x_1,\dots,x_m)$ is the Vandermonde matrix, i.e. $V_{ij}=x_i^{j-1}$. The inverse of the Vandermonde matrix has entries
$$ \bigl(V^{-1}\bigr)_{ij}=(-1)^{m-i}\frac{e_{m-i}(x_1,\dots,\widehat{x_j},\dots, x_m)}{\prod_{k\neq j}(x_j-x_k)},
$$
where $e_l$ denotes the $l$-th elementary symmetric polynomial (with $e_0=1$). It follows that if $Y(t)$ defines a curve in  $D_{2,m}(\mu_1,\mu_2)^o$ with a limit as $t\to 0$ in $D_{2,m}(\mu_1,\mu_2)$, then the corresponding $Y_i(t)$ in $D_{2,1}(\mu_1,\mu_2)$ satisfy 
\begin{equation} \lim_{t\to 0} \sum_{k=1}^m (V(t)^{-1})_{ik}Y_k(t) V_{kj}(t)\enskip \text{exists for all $i,j=1,\dots,m$}.\label{condition}
\end{equation}

If $u_1,\dots,u_m\in \cx$, then the $j$-th column of $V^{-1}\diag(u_1,\dots,u_m)V$ is given by the coefficients of the polynomial $p_j(x)$ of degree $\leq m-1$ satisfying
$p_j(x_k)=u_k x_k^{j-1}$. Thus the necessary and sufficient condition for the condition \eqref{condition} to be satisfied is that there exists polynomials $a(x)$ and $c(x)$ of degree $\leq m-1$ such that
\begin{equation*} Y_k=\begin{pmatrix} a(x_k) & \tau-x_kc(x_k)\\ c(x_k)& -a(x_k)\end{pmatrix},\enskip k=1,\dots, m, 
\end{equation*}
whenever the $x_k$ are distinct. Recalling the equation of $D_{2,1}(\mu_1,\mu_2)$ from the previous section, we conclude that the polynomials $a(x)$ and $c(x)$ satisfy the equation
\begin{equation} a(x)^2-x c(x) ^2+\frac{1}{4}x+(\mu_1^2-\mu_2^2)c(x)-\frac{1}{2}(\mu_1^2+\mu_2^2)=0\enskip\mod q(x) \label{pols}
\end{equation}
where $q(x)=\prod (x-x_k)$ whenever the $x_k$ are distinct. Since this equation extends to the case of non-distinct $x_k$, we conclude that:
\begin{theorem} The variety $D_{2,m}(\mu_1,\mu_2)$ is given by a monic polynomial $q(x)$ of degree $m$, and two polynomials $a(x),c(x)$ of degrees at most $ m-1$, such that \eqref{pols} is satisfied. 
\end{theorem} 
Comparing with \eqref{aff-Hilb} we conclude:
\begin{corollary} The variety $D_{2,m}(\mu_1,\mu_2)$ is isomorphic to the transverse Hilbert scheme $\bigl(D_{2,1}(\mu_1,\mu_2)\bigr)^{[m]}_\pi$ of $m$ points on the deformation of the $D_2$-singularity defined by the equation
$$ a^2-xc^2+\frac{1}{4}x+(\mu_1^2-\mu_2^2)c-\frac{1}{2}(\mu_1^2+\mu_2^2)=0,$$
with $\pi(a,c,x)=x$.
\end{corollary} 

We can also conclude from the above proof that the hyperk\"ahler metric on $D_{2,m}(\mu_1,\mu_2)$ is the one given by the fibrewise transverse Hilbert scheme construction, described in \S\ref{fibre}, applied to the twistor space of $D_{2,1}(\mu_1,\mu_2)$. Indeed, in the case $\mu_1=\mu_2$, applying the above calculations fibrewise, shows that $q(x)$ defines a curve $C$ of degree $m$ in the total space of $\sO(4)$ over $\oP^1$ and $a(x)+c(x)\sqrt{x}$ defines a section $s(\zeta,x)$ of $L^2(2)$ over $C$ such that $(s(\zeta,x))+(s(\zeta,-x))=(\mu^2(\zeta)-x$. Comparing with the description of the twistor space of $D_{2,1}(\mu,\mu)$, given in \S\ref{metric}, proves the claim. Similarly, in the case $\mu_1\neq \mu_2$, $xc(x)-\tau(x)/2-a(x)\sqrt{x}$ defines an appropriate section of $L^2(4)$, and again, the cmparison with \S\ref{metric} shows that both constructions produce the same hyperk\"ahler metric. 

\medskip

{\boldmath $D_{1,m}(\mu_1,\mu_2)$ and \boldmath $D_{0,m}(\mu_1,\mu_2)$}. We proceed similarly. Let $(S,Y)$ be an element of either of these varieties in the canonical form described in Proposition \ref{can}. The characteristic polynomial $Q(z)$ of $S_0$ still has the form  $\prod_{i=1}^m (z^2-x_i)$ and, if we assume that the $x_i$ are distinct, we can pass from the basis of $\cx[z]/(P(z))$
described in the proof of that Proposition to the basis consisting of polynomials 
\begin{equation} f_1(z),zf_1(z),f_2(z),zf_2(z),\dots, f_m(z),zf_m(z),\quad f_i(z)=\prod_{j\neq i} (z^2-x_j),
\label{f_inew}\end{equation}
 plus $f_{2m+1}=Q(z)$  in the $D_{1,m}$-case (resp. $f_{2m+1}= (2\mu_2)^{-1}(z-\mu_1+\mu_2)Q(z), f_{2m+2}=-(2\mu_2)^{-1}(z-\mu_1-\mu_2)Q(z)$ in the  $D_{0,m}$-case). In this new basis $S$ and $Y$ are still of the form \eqref{D1m} or \eqref{D0m} with $S_0$ block-diagonal with blocks \eqref{S_i} and the covector $e$ is equal to $(0,1,0,1,\dots,0,1)$. The matrix $Y_0$ is then, owing to the argument given above for $D_{2,m}(\mu_1,\mu_2)$, also block-diagonal. It follows that the minor matrices consisting of a single block of $S_0$, the corresponding entries of the last row (or the last two rows in the $D_{0,m}$-case) and the corresponding entries of the last column (or, again, last two columns in the   $D_{0,m}$-case), and the analogous minor matrix for $Y$ are elements of $D_{1,1}(\mu_1,\mu_2)$ or $D_{0,1}(\mu_1,\mu_2)$. Now,  analogously to the $D_{2,m}$-case, we change the basis back to the one given in the proof of Proposition \ref{can} and conclude that $D_{1,m}(\mu_1,\mu_2)$ or $D_{0,m}(\mu_1,\mu_2)$ are 
described by the same  affine equations which define $D_{1,1}(\mu_1,\mu_2)$ or $D_{0,1}(\mu_1,\mu_2)$ but this time in $\cx[x]/(q(x)$, where $q(x)=\prod_{i=1}^m(x-x_i)$. Comparing with \eqref{aff-Hilb} and with example \ref{AH-Hilb} yields:
\begin{theorem} The variety $D_{j,m}(\mu_1,\mu_2)$, $j=0,1$, is isomorphic to the transverse Hilbert scheme $\bigl(D_{j,1}(\mu_1,\mu_2)\bigr)^{[m]}_\pi$ of $m$ points on the $D_{j}$-surface defined by equation \eqref{D1} or \eqref{D0} with $\pi$ being the projection onto the $x$-coordinate. In particular,  $D_{1,m}(\mu,\mu)$ is biholomorphic to the space of rational maps of degree $2n$ of the form $p(u)/q(u^2)$, $\deg p=2n-1$, $\deg q=n$, and satisfying $p(u)p(-u)=1 \mod q(u^2)$. 
\end{theorem}

Once again, we can go through above calculations fibrewise on the twistor space of the $D_1$- or $D_0$-surface, and conclude that the hyperk\"ahler metric obtained from the slice construction of \cite{Bie1} coincides with the one obtained from the fibrewise transverse Hilbert scheme  construction described in \S\ref{fibre}. In particular, example \ref{tot} shows:
\begin{corollary} The hyperk\"ahler manifold  $D_{1,m}(\mu,\mu)$  is isometric to a totally geodesic submanifold of the moduli space of monopoles of charge $2n$, consisting of monopoles invariant under the involution $x\mapsto -x$ in $\oR^3$  and with total phase equal to $1$.\end{corollary}

\appendix
\section{Completeness of hyperk\"ahler slices}

Let $G$ be a compact Lie group, $\g$ its Lie algebra, and $\langle\:,\:\rangle$ an invariant scalar product on $\g$. Let $\rho:\su(2)\to \g$ be a homomorphism of Lie algebras, and write $\alpha_i$, $i=1,2,3$ for the images of the standard generators of $\su(2)$.
We consider quadruples of $\g$-valued smooth functions $T_i(t)$ on $(0,1]$ such that at $t=0$ $T_0$ is smooth, while $T_1,T_2,T_3$ have simple poles with residues $\alpha_i$. The Nahm equations are a system of ordinary differential equations, consisting of $ \dot{T}_1+[T_0,T_1]=[T_2,T_3]$ and two further equations given by cyclic permutations of indices $1,2,3$. The group $\sG$ of smooth gauge transformations $g:[0,1]\to G$, $g(0)=g(1)=1$, acts on the set $\sZ(\rho)$ of solutions having the above boundary conditions and the quotient is a finite-dimensional smooth manifold, which we denote by $N(\rho)$, and which is diffeomorphic to $S(\rho)\times G^\cx$ (see \S\ref{slices} and \cite{Bie1}). The natural $L^2$-metric (with respect to  $\langle\:,\:\rangle$) on the infinite-dimensional manifold $\sZ(\rho)$ is preserved by the group $\sG$ and it induces a Riemannian metric on $N(\rho)$. This metric is hyperk\"ahler and it is the one used in the hyperk\"ahler slice construction of Theorem \ref{theorem}. There are 
two tri-hamiltonian group actions on $N(\rho)$: the group $G$ acts by allowing gauge transformations with arbitrary values at $t=1$; the group $H\subset G$, the Lie algebra of which is the centraliser of $\rho(\su(2))$, acts by allowing  gauge transformations $g(t)$ with $g(0)\in H$. 
\par
We are going to prove:
\begin{theorem} The natural $L^2$-metric on $N(\rho)$ is complete.
\end{theorem}
\begin{proof}
Our first goal is a suitable description of $N(\rho)$ as an infinite-dimensional hyperk\"ahler quotient. We start with an affine space $\sM(\rho)$ of quadruples $(X_0,S_1+X_1,S_2+X_2,S_3+X_3)$ of $\g$-valued functions on $(0,1]$, where   $S_i(t)=\frac{\alpha_i}{t}$ and $X_j\in L^2\bigl((0,1),\g\bigr)$, $j=0,1,2,3$. It is a flat hyperk\"ahler Hilbert manifold (modelled on  $L^2\bigl((0,1),\g\bigr)\otimes \oR^4$) consisting of $\g$-valued quadruples $(T_0,T_1,T_2,T_3)$ with prescribed boundary behaviour. The relevant group $\sG^\prime$ of gauge transformation has the Lie algebra consisting of maps $\phi:[0,1]\rightarrow \g$  of class $W^{1,2}$  satisfying $\phi(0)=\phi(1)=0$. The corresponding fundamental vector field $\tilde\phi$ is
$$ \left(-\dot\phi+[\phi,T_0], [\phi,T_1], [\phi,T_2],[\phi,T_3]\right).$$
Since $\phi(t)$ belongs to $W^{1,2}(0,1)$, it has the form $\int_0^t\kappa(\tau)d\tau$ for an $L^2$-function $\kappa$. The Hardy inequality implies then that $\frac{\phi}{t}$ is square-integrable, so that the vector field $\tilde\phi$ is indeed in $L^2\bigl((0,1),\g\bigr)\otimes \oR^4$. The action of $\sG^\prime$ is free, proper, and isometric. These conditions suffice to conclude that the hyperk\"ahler quotient of $\sM(\rho)$ by the Hilbert Lie group $\sG^\prime$ is a hyperk\"ahler manifold. The zero-level set $\sZ(\rho)$ of the hyperk\"ahler moment map consists of weak solutions to Nahm's equations. The arguments in \cite{Bie1} are still valid and imply that the moduli space of weak solutions to Nahm's equations in $\sM(\rho)$ modulo the group the Hilbert Lie group $\sG^\prime$ is isometric to $N(\rho)$.
\par
In order to prove completeness of $N(\rho)$, we need the following lemma.
\begin{lemma} Any smooth curve $\gamma:[0,a)\to N(\rho)$ can be lifted to a horizontal curve in $\sZ(\rho)$.\end{lemma}
\begin{proof}  Since $N(\rho)$ can be alternatively described as the space of weak solutions to Nahm's equations modulo $W^{1,2}$ gauge transformations {\em or} the space of smooth solutions modulo smooth gauge transformations, we can find a smooth lift $T:[0,a)\to \sZ(\rho)$ of $\gamma$ consisting of smooth solutions to Nahm's equations. We seek a smooth map $g:[0,a)\to \sG$ such that $(g(s).T(s))^\prime$ is horizontal for each $s$. This means that $T^\prime(s)+(g^{-1}(s)g'(s))\tilde{}$ is horizontal for each $s$, where $\tilde\phi$ denotes the fundamental vector field corresponding to a $\phi\in \Lie \sG$. If we write $T(s)=(T_0(t,s),T_1(t,s), T_2(t,s),T_3(t,s))$ and $T^\prime (s)=(t_0(t,s),t_1(t,s), t_2(t,s),t_3(t,s))$, then a $\phi$ such that $T^\prime(s)+\tilde\phi$ is horizontal is a solution of the following linear differential equation
$$ \ddot{\phi}-2[\dot\phi,T_0]-[\phi,T_0]+\sum_{i=1}^3[T_i,[T_i,\phi]]=\dot{t}_0+\sum_{i=0}^3 [T_i,t_i],$$
with $\phi(0)=\phi(1)=0$ and the dot denoting derivation with respect to $t$. The solution is a smooth map $\phi:[0,1]\times [0,a)\to \g$. We can solve for each $t$ the linear equation $g^{-1}g^\prime=\phi$ in the Lie group $G$ and this produces the desired curve $g:[0,a)\to \sG$.\end{proof}

To finish the proof of completeness, let  $(m_k)_{k\in \oN}$ be a Cauchy sequence in $N(\rho)$ with corresponding representatives $T^k=(T_0^k,T_1^k,T_2^k,T_3^k)$ in $Z(\rho)$. The above lemma implies that $(T^k)$ is a Cauchy sequence in $\sM(\rho)$ (i.e. the corresponding $(X_0^k,X_1^k,X_2^k,X_3^k)$ constitute a Cauchy sequence in $L^2\bigl((0,1),\g\bigr)\otimes \oR^4$) and it has therefore a limit $T^\infty$ in $\sM(\rho)$. This limit is weak solution to Nahm's equations and hence $T^\infty\in \sZ(\rho)$. Since Riemannian submersions between Hilbert manifolds shorten distances,  $m_k$ converges in $N(\rho)$ to the equivalence class of $T^\infty$.
\end{proof}

\end{document}